\documentclass[11pt]{amsart}

\usepackage{amscd,amsmath,amssymb,amsfonts,verbatim}
\usepackage{graphicx}
\usepackage[active]{srcltx}
\usepackage[cmtip, all]{xy}
\usepackage[pdftex,colorlinks=true,citecolor=black,linkcolor=black,urlcolor=blue]{hyperref}

\newtheorem{theorem}{Theorem}[section]
\newtheorem{proposition}[theorem]{Proposition}
\newtheorem{lemma}[theorem]{Lemma}
\newtheorem{corollary}[theorem]{Corollary}

\newtheorem{question}[theorem]{Question}

\theoremstyle{definition}

\newtheorem{definition}[theorem]{Definition}

\theoremstyle{remark}
\newtheorem{remark}[theorem]{Remark}
\newtheorem*{ack}{Acknowledgments}
\numberwithin{equation}{section}

\providecommand{\MU}{\mathop{\rm MU}\nolimits}

\newcommand{\sE}{{\mathcal E}}

\newcommand{\sL}{{\mathcal L}}
\newcommand{\sM}{{\mathcal M}}
\newcommand{\sN}{{\mathcal N}}
\newcommand{\sO}{{\mathcal O}}

\newcommand{\sS}{{\mathcal S}}

\newcommand{\sU}{{\mathcal U}}

\newcommand{\sZ}{{\mathcal Z}}
\newcommand{\A}{{\mathbb A}}

\newcommand{\C}{{\mathbb C}}

\newcommand{\bL}{{\mathbb L}}

\renewcommand{\P}{{\mathbb P}}
\newcommand{\Q}{{\mathbb Q}}

\newcommand{\Z}{{\mathbb Z}}

\newcommand{\alg}{{\rm alg}}
\newcommand{\num}{{\rm num}}
\newcommand{\cl}{{\rm cl}}

\newcommand{\CH}{{\rm CH}}
\newcommand{\mc}{\mathcal}

\newcommand{\red}{{\rm red}}
\newcommand{\codim}{{\rm codim}}

\newcommand{\Hom}{{\rm Hom}}

\newcommand{\Spec}{{\rm Spec \,}}

\newcommand{\Sch}{{\operatorname{\mathbf{Sch}}}}

\renewcommand{\max}{{\operatorname{\rm max}}}

\newcommand{\Sm}{{\mathbf{Sm}}}
\newcommand{\SmProj}{{\mathbf{SmProj}}}

\newcommand{\Sym}{{\operatorname{\rm Sym}}}

\newcommand{\eff}{{\operatorname{\rm eff}}}

\newcommand{\et}{{\text{\'et}}}

\newcommand{\ds}{{/\kern-3pt/}}

\newcommand{\colim}{\mathop{\text{colim}}}

\newcommand{\lci}{{l.c.i.\ }}

\newcommand{\Cor}{{\operatorname{Cor}}}

\newcommand{\ov}{\overline}

\newcommand{\Jac}{{\rm Jac}}

\newcommand{\tuborg}{\left\{\begin{array}{ll}}
\newcommand{\sluttuborg}{\end{array}\right.}

\begin{document}

\title{On numerical equivalence for algebraic cobordism}
\author{Anandam BANERJEE  and Jinhyun PARK}
\address{Department of Mathematical Sciences, KAIST, 291 Daehak-ro, Yuseong-gu, Daejeon, 305-701, Republic of Korea (South)}
\email{anandam@mathsci.kaist.ac.kr; anandamb@gmail.com}
\address{Department of Mathematical Sciences, KAIST, 291 Daehak-ro, Yuseong-gu,
Daejeon, 305-701, Republic of Korea (South)}
\email{jinhyun@mathsci.kaist.ac.kr; jinhyun@kaist.edu}


\begin{abstract}We define and study the notion of numerical equivalence on algebraic cobordism cycles. We prove that algebraic cobordism modulo numerical equivalence is a finitely generated module over the Lazard ring, and it reproduces the Chow group modulo numerical equivalence. We show this theory defines an oriented Borel-Moore homology theory on schemes and oriented cohomology theory on smooth varieties.

We compare it with homological equivalence and smash-equivalence for cobordism cycles. For the former, we show that homological equivalence on algebraic cobordism is strictly finer than numerical equivalence, answering negatively the integral cobordism analogue of the standard conjecture $(D)$. For the latter, using Kimura finiteness on cobordism motives, we partially resolve the cobordism analogue of a conjecture by Voevodsky on rational smash-equivalence and numerical equivalence.
\end{abstract}


\maketitle

\section{Introduction}
In the theory of algebraic cycles, various adequate equivalences on them play essential roles (see \cite[\S~3.1--2]{Andre}). Now, the theory of algebraic cobordism, as pioneered by Levine and Morel in \cite{LM} (see also \cite{LP} and \S \ref{sec:cobordism}) allows one to study algebraic cycles and motives from a more general perspective. It is an interesting question to ask whether various adequate equivalences for algebraic cycles can be lifted up to the level of algebraic cobordism cycles. One such attempt for algebraic equivalence was made in \cite{KP} so that one obtained the theory $\Omega_{\alg} ^*(-)$ of algebraic cobordism modulo algebraic equivalence. 

The objective of the present work is to define and study the notion of numerical equivalence for algebraic cobordism cycles. In the classical situation of algebraic cycles, it has a long history and was the ground of ``enumerative geometry''. In simple words, this is based on counting the numbers of intersection points of two varieties, with suitable multiplicities. For cobordism cycles, this naive notion of counting does not work and this requires a bit of care.

This trouble can be overcome by treating the graded ring $\mathbb{L}$, called the Lazard ring (see \S \ref{sec:cobordism}), as a substitute for the ring $\mathbb{Z}$ of integers; instead of naively counting the number of intersection points, we ``count'' the $\mathbb{L}$-``intersection'' values of two cobordism cycles. Notice that such an idea of using $\mathbb{L}$ systematically instead of $\mathbb{Z}$ is not new. In \cite{LM},  $\mathbb{L}$ is the cobordism ring $\Omega^* (pt)$ of a point. In \cite{KP}, it was proved that the Griffiths group is a finitely generated abelian group if and only if the cobordism $\Omega^*_{\alg}$ is a finitely generated $\mathbb{L}$-module. Furthermore, the functor $- \otimes_{\mathbb{L}} \mathbb{Z}$ carries the cobordism groups to Chow groups with the respective adequate equivalences. 

We define numerical equivalence for algebraic cobordism cycles, and algebraic cobordism modulo numerical equivalence in \S \ref{sec:num} for smooth projective varieties over a field $k$ of characteristic $0$, exploiting the ring structure of $\Omega^* (X)$ as well as the push-forward $\pi_* : \Omega_* (X) \to \Omega_* (k)$ via the structure morphism $\pi: X \to \Spec (k)$.

When $X$ is a quasi-projective $k$-scheme, either $\Omega_* (X)$ may not be a ring, or there is no longer the push-forward via the structure map. However, using some standard lengthy arguments involving Hironaka desingularizations, we can naturally extend the functor $\Omega_* (X)$ defined on smooth projective varieties to the category $\Sch_k$ of quasi-projective $k$-schemes. This is done in \S \ref{sec:OBM}, and we prove that:

\begin{theorem}Let $k$ be a field of characteristic $0$.  The functor $\Omega_* ^{\num}$ on $\Sch_k$ defines an oriented Borel-Moore homology theory in the sense of \cite[Definition 5.1.3]{LM}, and the functor $\Omega^* _{\num}$ on $\Sm_k$, in cohomological indexing, defines an oriented cohomology theory in the sense of \cite[Definition 1.1.2]{LM}. In particular, $\Omega_* ^{\num}$ satisfies homotopy invariance, localization, and projective bundle formula. The theory $\Omega_* ^{\num}$ is generically constant, and satisfies the generalized degree formula in the sense of \cite[Theorem 4.4.7]{LM}.

We have natural surjective morphisms $\Omega_* \to \Omega_* ^{\alg} \to \Omega_* ^{\num}$ of oriented Borel-Moore homology theories on $\Sch_k$.
\end{theorem}

If one focuses on cellular varieties, using the above theorem we obtain:

\begin{theorem}Let $X$ be a cellular variety, and let $r$ be the number of all cells. Then, the natural surjections $\Omega_* (X) \to \Omega_* ^{\alg} (X) \to \Omega_* ^{\num} (X)$ are isomorphisms of $\mathbb{L}$-modules, all isomorphic to the free $\mathbb{L}$-module of rank $r$. 
\end{theorem}

One important aspect of numerical equivalence for algebraic cycles is its finiteness property. We have its cobordism analogue, too:

\begin{theorem}Let $X\in \Sch_k$, where $k$ is a field of characteristic zero. Then, $(1)$ we have an isomorphism $\Omega_* ^{\num} (X) \otimes_{\mathbb{L}} \mathbb{Z} \simeq \CH_* ^{\num} (X)$, where $\CH_* ^{\num}$ denotes the Chow group of $X$ modulo numerical equivalence, and $(2)$ $\Omega_* ^{\num} (X)$ is a finitely generated $\mathbb{L}$-module.
\end{theorem}

For algebraic cycles, to each choice of a Weil cohomology theory, there corresponds a homological equivalence. For cobordism cycles, a notion of homological equivalence was introduced in \cite[\S 9]{KP} using complex cobordism as a cohomology theory, when $X$ is a smooth variety over $\mathbb{C}$. For more general case, one may use the theory of \'etale cobordism of \cite{Quick}. In \S \ref{sec:hom}, we first prove:

\begin{theorem}\label{thm:D_forward_intro}Let $X$ be a smooth projective variety over $\mathbb{C}$. Then, a homologically trivial cobordism cycle is also numerically trivial.
\end{theorem}

The classical standard conjecture $(D)$ expects that homological equivalence is equivalent to numerical equivalence for algebraic cycles with $\mathbb{Q}$-coefficients. For cobordism cycles, we can ask whether the converse of Theorem \ref{thm:D_forward_intro} holds. Let's call it the statement $(D)_{\MU}$. It turns out that, this is related to the cycle class map $cl_T: \CH^* (X) \to \MU^* (X) \otimes_{\mathbb{L}} \mathbb{Z}$ of Totaro in \cite[Theorems 3.1, 4.1]{Totaro}. The statement $(D)_{\rm T}$ is that $\ker (cl_T) = {\rm Num}^* (X)$, where ${\rm Num} ^* (X)$ is the group of numerically trivial cycles. The statement $(D)_{\rm H}$ is that $\ker (cl) = {\rm Num} ^* (X)$ for the usual cycle class map $cl: \CH^* (X) \to H^{*} (X, \mathbb{Z})$. Then, we have:

\begin{theorem}\label{thm:standard_intro}Let $X$ be a smooth projective variety over $\mathbb{C}$. For the various statements of standard conjectures, $(1)$ $(D)_{\MU}$ is equivalent to $(D)_{\rm T}$, $(2)$ $(D)_{\rm T}$ is equivalent to $(D)_{\rm H}$ when the natural map $\MU^* (X) \otimes_{\mathbb{L}} \mathbb{Z} \to H^* (X, \mathbb{Z})$ is an isomorphism, and $(3)$ all three statements are equivalent with $\mathbb{Q}$-coefficients.
\end{theorem}

When one works with algebraic cycles with $\mathbb{Q}$, one approach for the standard conjecture $(D)$ suggested by Voevodsky (\cite[Conjecture 4.2]{Voevodsky}) is to use the (rational) smash-nilpotence on algebraic cycles, first considered in \cite{Voevodsky} and \cite{Voisin}. More specifically, he conjectured that algebraic cycles with $\mathbb{Q}$-coefficients are rationally smash-nilpotent if and only if it is numerically trivial. Since we do not take $\mathbb{Q}$-coefficients at this moment, we can formulate the integral cobordism analogue of the Voevodsky conjecture: \emph{is a numerically trivial cobordism cycle, also (integrally) smash-nilpotent?} This conjecture is stronger than the standard conjecture $(D)_{\MU}$. Interestingly, combining Theorem \ref{thm:standard_intro} with a celebrated example in \cite{Totaro}, we deduce:

\begin{theorem}There is a smooth projective variety over $\mathbb{C}$ for which the standard conjecture $(D)_{\MU}$ fails. In particular, the homological equivalence is strictly finer than the numerical equivalence in general. It implies that the integral cobordism analogue of the Voevodsky conjecture also fails in general.
\end{theorem}

Although the integral cobordism analogue of the Voevodsky conjecture is false, one may still ask if it works with the $\mathbb{Q}$-coefficients. The question on the rational cobordism analogue of Voevodsky conjecture was first raised in \cite[Remark~10.5]{KP}. A benefit of considering this question lies in that we do not need to worry about the choice of the cohomology theory such as $\MU$ in the definition of the homological equivalence. The rest of the paper is an attempt to check this for certain cases known for algebraic cycles. Recently, Sebastian \cite{Sebastian} proved part of the Voevodsky conjecture for $1$-cycles on smooth projective varieties dominated by products of curves. We prove the forward direction in \S \ref{sec:smash}, and the cobordism analogue of the result of Sebastian is answered in \S \ref{sec:Sebastian}. In the process, we develop the notion of Kimura finiteness on cobordism motives in \S \ref{sec:Kimura}. Here is a summary:

\begin{theorem}Let $X$ be a smooth projective variety over $k$ of characteristic $0$ and let $\alpha$ be a cobordism cycle over $X$. $(1)$ If $k$ has an embedding into $\C$ and $\alpha$ is smash nilpotent, then it is homologically trivial. $(2)$ If $\alpha$ is smash nilpotent, then it is numerically trivial. $(3)$ If $k = \bar{k}$ and $X$ is dominated by a product of curves, then any numerically trivial cobordism $1$-cycle $\alpha$ is smash nilpotent.
\end{theorem}

\noindent \textbf{Conventions:} We always suppose that the base field $k$ is of characteristic $0$. A $k$-scheme, or a scheme always means a quasi-projective schme of finite type over $k$, and a $k$-variety means an integral $k$-scheme. The category of $k$-schemes is denoted by $\Sch_k$ and the category of smooth $k$-schemes is denoted by $\Sm_k$ throughout the paper.

\section{Recollection of algebraic cobordism theories}
\label{sec:cobordism}
We recall some definitions on algebraic cobordism from \cite[\S 2.1--4]{LM}. Let $X$ be a $k$-scheme of finite type. 

\begin{definition}[{\cite[Definition~2.1.6]{LM}}]
A \emph{cobordism cycle} over $X$ is a family $(f:Y\to X,L_1,\ldots,L_r)$, where $Y$ is smooth and integral, $f$ is
projective, and $(L_1 , \ldots , L_r)$ is a finite sequence of $r \geq 0$ line bundles over $Y$. Its dimension is defined to be $\dim(Y)-r \in \Z$. An isomorphism $\Phi$ of cobordism cycles $(Y\to X,L_1,\ldots,L_r)\overset{\sim}{\to} (Y'\to X,L'_1,\ldots,L'_r)$
 is a triple $\Phi=(\phi:Y\to Y',\sigma,(\psi_1,\ldots,\psi_r))$ consisting of an 
isomorphism $\phi:Y\to Y$ of $X$-schemes, a bijection $\sigma : \{1, \ldots , r\} \overset{\sim}{\to} \{1, \ldots , r\}$, and
isomorphisms $\psi_i:L_i\overset{\sim}{\to} \phi^*L'_{\sigma(i)}$
 of lines bundles over $Y$ for all $i$. 
\end{definition}

Let $\sZ(X)$ be the free abelian group on the set of isomorphism classes of cobordism cycles over $X$. 
Grading by the dimension of cobordism cycles makes $\sZ_*(X)$ into a graded abelian group. The image of a cobordism cycle $(Y\to X,L_1,\ldots,L_r)$ in
$\sZ_*(X)$ is denoted by $[Y\to X,L_1,\ldots,L_r]$. When X is smooth and equidimensional,
the class $[{\rm Id}_X : X \to X] \in \sZ_d(X)$ is denoted often as $1_X$. 

\begin{definition}[{\cite[\S 2.1.2--3]{LM}}]$\;$\\
\vspace*{-6mm}\begin{enumerate}
\item For a projective morphism $g : X \to X'$ in $\Sch_k$, composition with $g$ defines the graded group homomorphism $g_* : \sZ_*(X)\to\sZ_*(X')$ given by $[f : Y \to X, L_1 , \ldots , L_r ] $ $\mapsto $ $[g \circ f : Y \to X' , L_1 , \ldots , L_r ]$, called the \emph{push-forward along $g$}.

\item  If $g : X \to X'$ is a smooth equidimensional morphism of relative dimension $d$,
the \emph{pull-back along $g$} is defined to be the homomorphism $
g^* : \sZ_*(X')\to\sZ_{*+d}(X)$, 
$[f : Y \to X' , L_1 , \ldots , L_r ]$ $\mapsto$ $[pr_2 : Y \times_{X'} X \to X, pr_1^* (L_1 ), \ldots , pr_1^* (L_r )].$

\item  Let $L$ be a line bundle on $X$. The homomorphism $\widetilde{c}_1(L) : \sZ_*(X)\to\sZ_{*-1}(X)$ defined by $[f : Y \to X, L_1 , \ldots , L_r ]$ $ \mapsto$ $[f : Y \to X, L_1 , \ldots , L_r , f^*(L)]$ is called the \emph{first Chern class operator of $L$}. If $X$ is smooth, the first Chern class $c_1(L)$ of
$L$ is defined to be the cobordism cycle $c_1(L) : = \widetilde{c}_1(L)(1_X)$.

\item  The \emph{external product} $\times : \sZ_*(X) \times \sZ_*(Y) \to \sZ_*(X \times Y )$ on the functor $\sZ_*$ is defined by
$[f : X' \to X, L_1 , \ldots , L_r ] \times [g : Y' \to Y, M_1 , \ldots , M_s ]$ $ \mapsto $
$[f \times g : X' \times Y' \to X \times Y, pr_1^*(L_1 ), \ldots , pr_1^*(L_r ), pr_2^*(M_1 ), \ldots , pr_2^*(M_s )].$
\end{enumerate}
\end{definition}

While for the Chow ring we have $c_1(L\otimes M)=c_1(L)+c_1(M)$, this is not always true for oriented cohomology theories (see \cite[Definition~1.1.2]{LM}) and addition has to be replaced by a formal group law: $c_1(L\otimes M)=F(c_1(L),c_1(M))$ for some power series $F$ in two variables. A commutative formal group law $(R,F_R)$ of rank $1$ consists of a ring $R$ and $F_R\in R[[u,v]]$ satisfying conditions analogous to the operations in a group. 
In \cite{Lazard}, Lazard showed that there exists a formal group law $(\bL,F_{\bL})$ of rank $1$ which is universal: for any other law $(R, F_R)$ there exists a unique morphism $\Phi_{(R, F_R)}:\bL\to R$ which maps the coefficients of $F_{\bL}$ onto those of $F_R$. The ring $\bL$, called the \emph{Lazard ring}, is isomorphic to the polynomial ring $\Z[a_{i}\vert i\geq 1]$, and can be made into a graded ring $\bL_*$ by assigning $\deg a_{i}=i$.
 See \cite[Section~1.1]{LM} for details. 
 
\begin{definition}[{\cite[Definitions~2.4.5,~2.4.10]{LM}}]
For $X \in \Sch_k$, algebraic cobordism $\Omega_*(X)$ is defined to be the quotient of $\sZ_*(X)\otimes\bL_*$ by the following three relations:
\item ${\rm (Dim)}$ If there is a smooth quasi-projective morphism $\pi:Y \to Z$ with line bundles $M_1 , \ldots , M_{s>\dim Z}$ on $Z$ with $L_i \overset{\sim}{\to} \pi^*M_i$ for $i=1,\ldots,s\leq r$, then $[f:Y\to X,L_1,\ldots,L_r ] = 0$.

\item ${\rm (Sect)}$ For a section $s : Y \to L$ of a line bundle $L$ on $Y$ with the associated smooth 
divisor $i : D \to Y$, we impose $[f : Y \to X, L_1 , \ldots , L_r , L] = [f \circ i : D \to X, i^*L_1 , \ldots , i^*L_r ].$

\item ${\rm (FGL)}$ For line bundles $L$ and $M$ on $X$, we impose the equality
$F_{\bL}(\widetilde{c}_1(L),\widetilde{c}_1(M))([f : Y \to X, L_1 , \ldots , L_r ]) =\widetilde{c}_1(L\otimes M)([f : Y \to X, L_1 , \ldots , L_r ]).$
By the relation ${\rm (Dim)}$, the
expression $F_{\bL}(\widetilde{c}_1(L),\widetilde{c}_1(M))$ is a finite sum so that the operator is well-defined.
\end{definition}
When $X$ is smooth and equidimensional of dimension $n$, the codimension of a cobordism $d$-cycle is defined to be $n-d$. We set 
$\Omega^{n-d}(X) : = \Omega_d(X)$, and $\Omega^*(X)$ is the direct sum of the groups over all codimensions. Levine and Morel showed (\cite[Theorem 1.2.6]{LM}) that algebraic cobordism $\Omega^*$ is a universal oriented cohomology theory on $\Sm_k$ in the sense of \cite[Definition 1.1.2]{LM}) when $k$ is a field of characteristic $0$.

\section{Numerical equivalence on cobordism cycles}\label{sec:num}

Let $X$ be a smooth projective variety over a field $k$ of characteristic 0. Consider the composition of maps 
\begin{equation}\label{eq:cobprod}
\Omega_*(X)\otimes\Omega_*(X)\overset{\times}{\to}\Omega_*(X\times X)\overset{\Delta^*_X}{\to}\Omega_*(X)\overset{\pi_*}{\to}\Omega_*(k),
\end{equation}  
where $\times$ is the external product of cobordism cycles, $\Delta_X^*$ is the l.c.i. pull-back by diagonal morphism as in \cite[\S 6.5.4]{LM}, and $\pi$ is the structure morphism $X\to\Spec(k)$. This gives a map of $\bL$-modules $\Omega_*(X)\longrightarrow \Hom_{\bL}(\Omega_*(X),\Omega_*(k))$. For $\alpha, \beta \in \Omega_* (X)$, we often write $\alpha \cdot \beta$ for $\Delta_X ^* (\alpha \times \beta)$.

\begin{definition}\label{def:numSmProj}
We say that a cobordism cycle is {\em numerically equivalent to} 0 if it is in the kernel of this map, $\mc{N}_*(X):=\ker (\Omega_*(X)\rightarrow \Hom_{\bL}(\Omega_*(X),\Omega_*(k)))$, and we let $\Omega_*^{\num}(X):=\Omega_*(X)/\mc{N}_*(X)$, which is algebraic cobordism modulo numerical equivalence.
\end{definition}

It is immediate from the definition that we have:

\begin{lemma}\label{lem:ideal_smp}The subgroup $\mathcal{N}_* (X) \subset \Omega_* (X)$ is an ideal.
\end{lemma}

\begin{remark}\label{rmk:point}
Note that when $X= \Spec (k)$, $\Omega_* (X)$ is just the Lazard ring $\mathbb{L}$, and the product \eqref{eq:cobprod} is just the product of the ring $\mathbb{L}$. Since $\mathbb{L}$ is a polynomial ring in infinitely many variables, it has no zero-divisor. In particular, $\mathcal{N}_* (\Spec (k)) = 0$. Hence, we deduce that $\Omega_* ^{\num} (\Spec (k)) = \mathbb{L}$.
\end{remark}

For the canonical map $\phi:\Omega_*(X)\to \CH_*(X)$ commutes with pull-backs and push-forwards and respects the product on the Chow group, one easily checks that $\phi$ maps $ \mathcal{N}_* (X)$ into the group $\mathrm{Num}_*(X)$ of algebraic cycles numerically equivalent to $0$. This gives a well-defined map $\phi^\num:\Omega^\num_*(X)\longrightarrow \CH^\num_*(X).$

Recall that $\phi$ factors through the canonical morphism $\Omega_*(X)\otimes_{\bL}\Z \overset{\overline{\phi}}{\to} \CH_*(X)$, which is an isomorphism of Borel-Moore weak homology theories on $\Sch_k$ (see \cite[Definition~4.1.9]{LM}). 

\begin{theorem}\label{thm:num}
Let $X$ be a smooth projective variety over a field $k$ of characteristic $0$.
\begin{enumerate}
\item $\phi^\num$ induces an isomorphism $ \overline{\phi}^\num\;:\;\Omega^{\num} _* (X) \otimes_{\mathbb{L}} \mathbb{Z} \simeq \CH_* ^{\num} (X)$.

\item $\Omega^{\num}_* (X)$ is a finitely generated $\mathbb{L}$-module.
\end{enumerate}
\end{theorem}

\begin{proof}
(1) Consider the following commutative diagram with exact rows:
\begin{equation}\label{eq:num}
\xymatrix{0 \ar[r] & \mc{N}_*(X)\ar[r] \ar[d] & \Omega_*(X) \ar[r] \ar[d]_{\phi}  & \Omega_*^\num(X) \ar[r] \ar[d]_{\phi^\num} & 0 \\
0 \ar[r] & \mathrm{Num}_*(X) \ar[r]  & \CH_*(X) \ar[r]  & \CH^{\num}_*(X) \ar[r]  & 0.}
\end{equation}Since $\phi$ is surjective by \cite[Theorem 1.2.19]{LM}, we immediately deduce that $\phi^{\num}$ is surjective from the right square.

Let's compute $\ker (\phi ^{\num})$. For $\alpha \in \Omega_* (X)$, let $\ov{\alpha} \in \Omega_* ^{\num} (X)$ be its image. Suppose $\ov{\alpha} \in \ker (\phi^{\num})$. Then, by \eqref{eq:num}, $\phi (\alpha) \in {\rm Num}_* (X)$. Since $\ker (\phi) = \mathbb{L}_{>0} \cdot \Omega_* (X)$ by \emph{loc. cit.}, we have $\alpha \in \mathcal{N}_* (X) + \mathbb{L}_{>0} \cdot \Omega_* (X)$. Thus, modulo $\mathcal{N}_* (X)$, we have $\ov{\alpha} \in \mathbb{L}_{>0} \cdot \Omega_* ^{\num} (X)$, i.e. $\ker (\phi^{\num}) \subset \mathbb{L}_{>0} \cdot \Omega_* ^{\num} (X)$. That $\ker (\phi^{\num}) \supset \mathbb{L}_{>0} \cdot \Omega_* ^{\num} (X)$ is obvious. Hence, the induced homomorphism $\ov{\phi}^{\num}: \Omega_* ^{\num} (X) \otimes_{\mathbb{L}} \mathbb{Z} \to \CH_* ^{\num} (X)$ is an isomorphism because $\Omega_* ^{\num} \otimes_{\mathbb{L}} \mathbb{Z} \simeq \Omega_* ^{\num} (X) / \mathbb{L}_{>0} \cdot \Omega_* ^{\num} (X) = \Omega_* ^{\num} (X) / \ker (\phi^{\num})$.

(2) 
This follows easily from \cite[Lemma~9.8]{KP} since $\CH^* _{\num} (X)$ is finitely generated. In fact, $\Omega^*_{\num}(X)$ is generated by any set of elements that map via $\phi$ to a set of generators of $\CH^* _{\num} (X)$.
\end{proof}

\begin{corollary}\label{cor:stpdstc_hn}
Let $X$ be a smooth projective variety over a field $k$ of characteristic $0$. Then,  $\mathcal{N}^* (X) \otimes_{\mathbb{L}} \mathbb{Z} = {\rm Num}^* (X)$.
\end{corollary}

\begin{proof}That $\mathcal{N}_* (X) \otimes_{\mathbb{L}} \mathbb{Z} \subset {\rm Num}_* (X)$ is obvious because the map $\phi$ in \eqref{eq:num} is ${\rm Id} \otimes_{\mathbb{L}} \mathbb{Z}= {\rm Id} / \mathbb{L}_{>0} \cdot {\rm Id}$. For the other inclusion, let $\beta \in {\rm Num}^* (X)$. Since $\phi$ is surjective, there is $ \alpha \in \Omega_* (X)$ such that $\phi (\alpha) = \beta$. But, as seen in Theorem \ref{thm:num}(1), we have $\alpha \in \mathcal{N}_* (X) + \mathbb{L}_{>0} \cdot \Omega_* (X)$, i.e. $\alpha = \alpha ' + \alpha ''$ for some $\alpha ' \in \mathcal{N}_* (X)$ and $\alpha'' \in \mathbb{L}_{>0} \cdot \Omega_* (X)$. Since $\phi$ annihilates all of $\mathbb{L}_{>0} \cdot \Omega_* (X)$, $\phi (\alpha) = \phi (\alpha ') = \beta$, so by replacing $\alpha$ by $\alpha'$, we may assume $\alpha \in \mathcal{N}_* (X)$. Then, $\beta \in \phi (\mathcal{N}_* (X)) = \mathcal{N}_* (X) \otimes_{\mathbb{L}} \mathbb{Z}$, finishing the proof.
\end{proof}

Let $X$ be a smooth projective irreducible $k$-scheme. The degree of a cobordism cycle on $X$ has been defined in \cite[Definition~4.4.4]{LM} to be a homomorphism $\deg : \Omega_*(X)\to\Omega_{*-\dim_kX}(k)$. 
\begin{proposition}
The degree of a numerically trivial cobordism cycle on $X$ is zero.
\end{proposition}

\begin{proof}
Let $\alpha\in \sN_*(X)$. By the generalized degree formula (\cite[Theorem 4.4.7]{LM}), for each closed integral subscheme $Z\subset X$, we have a projective birational morphism $\widetilde{Z}\to Z$ with $\widetilde{Z}$ in $\Sm_k$ and an element $\omega_Z\in \Omega_{*-\dim_kZ}(k)$, all but finitely many being zero, such that
\[ \alpha=\deg(\alpha)[{\rm Id}_X]+\sum_{\codim_X Z>0}\omega_Z\cdot [\widetilde{Z}\to X]. \]
By the definition of numerical triviality, for any $\gamma\in \Omega_*(X)$, $\pi_*(\alpha\cdot\gamma)=0$. Thus, $\deg(\alpha)\pi_*(\gamma)+{\sum_{\codim_XZ>0}\omega_Z\cdot\pi_*([\widetilde{Z}\to X]\cdot\gamma)}=0$. In particular, choose a $\gamma\in\Omega_0(X)$ such that $\pi_*(\gamma)\neq 0$. Then, $[\widetilde{Z}\to X]\cdot\gamma=0$ for all $Z$ with $\codim_XZ>0$. Since $\pi_*(\gamma)\neq 0$, we must have $\deg(\alpha)=0$.
\end{proof}

\section{$\Omega^\num_*$ as an oriented Borel-Moore homology theory}
\label{sec:OBM}

Let $k$ be a field of characteristic $0$. The objective of this section is to extend the functor $\Omega_* ^{\num}$ to $\Sch_k$ so as to obtained functor is an oriented Borel-Moore homology theory and study some of its basic properties. We begin with the following basic results:

\begin{lemma}\label{stpdab1}Let $f: X \to Y$ be a morphism of smooth projective varieties. Then, the l.c.i. pull-back $f^*: \Omega_* (Y) \to \Omega_{*+d} (X)$ as in \cite[\S 6.5.4]{LM}, where $d= \dim X - \dim Y$, maps $\mathcal{N}_* (Y)$ into $\mathcal{N}_* (X)$.
\end{lemma}

\begin{proof}Let $\alpha \in \mathcal{N}_* (Y)$. Let $\pi_X: X \to \Spec (k)$ and $\pi_Y: Y \to \Spec (k)$ be the structure maps. Since $X$ and $Y$ are both projective, $f$ is projective so that $f_*: \Omega_* (X) \to \Omega_* (Y)$ exists. Furthermore, $\pi_X = \pi_Y \circ f$ implies $\pi_{X*} = \pi_{Y*} f_*$. Hence, for any $\gamma \in \Omega_* (X)$, we have $\pi_{X*} (f^* \alpha \cdot \gamma) = \pi_{Y*} f_* (f^* \alpha \cdot \gamma) =^{\dagger} \pi_{Y*} (\alpha \cdot f_* \gamma) = 0$, where $\dagger$ holds by the projection formula. Hence, $f^* \alpha \in \mathcal{N}_* (X)$, i.e. $f^* (\mathcal{N}_* (Y)) \subset \mathcal{N}_* (X)$. 
\end{proof}

\begin{lemma}\label{stpdab1'}Let $f: X \to Y$ be a morphism of smooth projective varieties. Then, the projective push-forward $f_*: \Omega_* (X) \to \Omega_* (Y)$ maps $\mathcal{N}_* (X)$ into $\mathcal{N}_* (Y)$.
\end{lemma}

\begin{proof}Let $\alpha \in \mathcal{N}_* (X)$.  Then, for any $\gamma \in \Omega_* (Y)$, we have $0 = \pi_{X*} (\alpha \cdot f^* \gamma)$, where $f^*$ is the l.c.i. pull-back as in \cite[\S 6.5.4]{LM}. Note that $\pi_Y \circ f = \pi_X$ so that $\pi_{Y*} f_* = \pi_{X*}$. Hence, $0 = \pi_{Y*} f_* (\alpha \cdot f^* \gamma) =^{\dagger} \pi_{Y*} (f_* \alpha \cdot \gamma)$, where $\dagger$ holds by the projection formula. This means, $f_* \alpha \in \mathcal{N}_* (Y)$, i.e. $f_* \mathcal{N}_* (X) \subset \mathcal{N}_Y (Y)$ as desired.
\end{proof}

\begin{lemma}\label{stpdab3}Let $X$ be a quasi-projective $k$-variety. Let $f_{\ell} : \widetilde{X}_{\ell} \to X$ be desingularizations of $X$ for $\ell = 1,2$. Then, there is another desingularization $f_3: \widetilde{X}_3 \to X$ and morphisms $g_\ell : \widetilde{X}_3 \to \widetilde{X}_\ell$ such that $f_3 = f_{\ell} \circ g_{\ell} $ for $\ell = 1,2$.
\end{lemma}

\begin{proof}This is standard: let $i: X_{\rm sm} \subset X$ be the smooth locus of $X$ and let $j_{\ell} : X_{\rm sm} \hookrightarrow \widetilde{X}_{\ell}$ be the open immersions such that $i = f_{\ell} \circ j_{\ell}$ for $\ell = 1,2$. Let $\bar{\Delta}$ be the Zariski closure of the image of $(j_1, j_2): X_{\rm sm} \to \widetilde{X}_1 \times \widetilde{X}_2$, and consider the compositions $g_{\ell}: \widetilde{X}_3 \to \bar{\Delta} \hookrightarrow \widetilde{X}_1 \times \widetilde{X}_2 \overset{p_\ell}{\to} \widetilde{X}_\ell$ for $\ell =1, 2$, where the first arrow is a desingularization and $p_{\ell}$ is the projection. Define $f_3= f_1 \circ g_1$. Since $f_1 \circ p_1 = f_2 \circ p_2$, we have $f_1 \circ g_1 = f_2 \circ g_2$. Thus, we are done.
\end{proof}

\subsection{Extension to quasi-projective schemes}\label{subsec:definition} 

\subsubsection{Smooth quasi-projective varieties} As the first step, let $U$ be a smooth quasi-projective $k$-variety. Then for some open immersion $i : U \hookrightarrow X$, with $X$ projective, the singular locus $X_{\rm sing}$ lies in $X \setminus U$. Let $f: \widetilde{X} \to X$ be a desingularization. Let $j: U\hookrightarrow \widetilde{X}$ be the open immersion such that $f \circ j = i$. We call such $j$ a smooth compactification. Since $j$ is a morphism between two smooth varieties, it is l.c.i., thus $j^*$ exists on cobordism cycles by \cite[\S 6.5.4]{LM}.

\begin{lemma}\label{lem:smqp}
For two smooth compactifications $j_\ell : U \hookrightarrow \widetilde{X}_{\ell}$, $\ell = 1,2$, we have $j_1 ^* \mathcal{N}_* (\widetilde{X}_1) = j_2 ^* \mathcal{N}_* (\widetilde{X}_2)$ in $\Omega_*(U)$.
\end{lemma}

\begin{proof}By Lemma \ref{stpdab3}, there is another smooth compactification $j_3 : U \hookrightarrow \widetilde{X}_3$ with morphisms $g_{\ell} : \widetilde{X}_3 \to \widetilde{X}_\ell$ such that $g_{\ell} \circ j_3 = j_\ell$ for $\ell=1,2$. 
So, by replacing $\widetilde{X}_2$ by $\widetilde{X}_3$, we may assume there is a birational projective morphism $g: \widetilde{X}_2 \to \widetilde{X}_1$. This satisfies $j_1 = g\circ j_2$. Since $g$ is l.c.i., we have $g^*$ on cobordism cycles, and it satisfies $j_1 ^* = j_2 ^* g^*$. Furthermore, we know $g^* (\mathcal{N}_* (\widetilde{X}_1)) \subset \mathcal{N}_* (\widetilde{X}_2)$ by Lemma \ref{stpdab1}. Now, let $\alpha \in \mathcal{N}_* (\widetilde{X}_1)$. Then, $j_1 ^* \alpha = j_2 ^* g^* \alpha \in j_2 ^* \mathcal{N}_* (\widetilde{X}_2)$. This shows $j_1 ^* \mathcal{N}_* (\widetilde{X}_1) \subset j_2 ^* \mathcal{N}_* (\widetilde{X}_2)$. 

For the other inclusion, for $\alpha \in \mathcal{N}_* (\widetilde{X}_2)$ and $\gamma \in \Omega_* (\widetilde{X}_1)$, we have $0 = \pi_{\widetilde{X}_2*} (\alpha \cdot g^* \gamma) = \pi_{\widetilde{X}_1*} g_* (\alpha \cdot g^* \gamma) =^{\dagger} \pi_{\widetilde{X}_1*} (g_* \alpha \cdot \gamma)$, where $\dagger$ holds by projection formula. Hence $g_* \alpha \in \mathcal{N}_* ( \widetilde{X}_1 )$. On the other hand, from the Cartesian square
\[\xymatrix{ U \ar@{^{(}->}[r]^{j_2} \ar@{=}[d] & \widetilde{X}_2 \ar[d]^{g} \\ U \ar@{^{(}->}[r]^{j_1} & \widetilde{X}_1 }, \]
for $\alpha \in \mathcal{N}_* (\widetilde{X}_2)$, we have $j_2 ^* \alpha = j_1 ^* g_* \alpha$ by \cite[Theorem 6.5.12]{LM}. Thus, $j_2 ^* \alpha = j_1 ^* g_* \alpha \in j_1 ^* \mathcal{N}_* (\widetilde{X}_1)$. This shows $j_1 ^* \mathcal{N}_* (\widetilde{X}_1) = j_2 ^* \mathcal{N}_* (\widetilde{X}_2)$ as desired.
\end{proof}

\begin{definition}\label{defn:smqp}
When $U$ is a smooth quasi-projective $k$-variety, define $\Omega_* ^{\num} (U):= \Omega_* (U) /\mathcal{N}_* (U)$, where $\mathcal{N}_* (U):= j^* \mathcal{N}_* (\widetilde{X})$, for any choice of smooth compactification $j: U \hookrightarrow \widetilde{X}$. By Lemma \ref{lem:smqp}, this is well-defined.
\end{definition}

We can strengthen Lemma \ref{lem:ideal_smp}:
\begin{lemma}\label{lem:ideal_smqp}The subgroup $\mathcal{N}_* (U)$ is an ideal of $\Omega_* (U)$.
\end{lemma}

\begin{proof}
Choose a smooth compactification $j: U \hookrightarrow \widetilde{X}$. Since $j^*: \Omega_* (\widetilde{X}) \to \Omega_* (U)$ is a surjective ring homomorphism by localization theorem \cite[Theorem 3.2.7]{LM}, and since $\mathcal{N}_* (\widetilde{X})$ is an ideal of $\Omega_* (\widetilde{X})$ by Lemma \ref{lem:ideal_smp}, by the correspondence theorem we deduce that $j^* \mathcal{N}_* (\widetilde{X})= \mathcal{N}_* (U)$ is an ideal of $\Omega_* (U)$.
\end{proof}

\begin{lemma}\label{stpdab4}Let $f: X \to Y$ be a morphism of smooth quasi-projective $k$-varieties. Then, there are smooth compactifications $j_1 : X \hookrightarrow \widetilde{X}$ and $j_2: Y \hookrightarrow \widetilde{Y}$ and a morphism $\widetilde{f}: \widetilde{X} \to \widetilde{Y} $ such that $\widetilde{f} \circ j_1 = j_2 \circ f$, and it is a Cartesian square. Furthermore, $j_2$ and $\widetilde{f}$ are Tor-independent.
\end{lemma}

\begin{proof}First, choose smooth compactifications $j_1 : X \hookrightarrow \widetilde{X}$ and $j_2: Y \hookrightarrow \widetilde{Y}$. Any desingularization $\widetilde{X}' \to \ov{\Gamma}$ of the Zariski closure of the graph $\Gamma$ of $f$ in $\widetilde{X} \times \widetilde{Y}$, gives another smooth compactification of $X$. Furthermore, we have a natural composition $ \widetilde{X}' \to \ov{\Gamma} \hookrightarrow \widetilde{X} \times \widetilde{Y} \overset{pr}{\to} \widetilde{Y}$, where $pr$ is the projection. So, replacing $\widetilde{X}$ by $\widetilde{X}'$, we have a desired morphism $\widetilde{f}: \widetilde{X} \to \widetilde{Y}$ such that $\widetilde{f} \circ j_1 = j_2 \circ f$. Tor-independence is obvious because $j_2$ is flat.
\end{proof}

We now strengthen Lemma \ref{stpdab1} a bit:

\begin{lemma}\label{lem:sm2pullbk}
Let $f: X \to Y$ be a morphism of smooth quasi-projective varieties. Then, the l.c.i. pull-back $f^*: \Omega_* (Y) \to \Omega_{*+d} (X)$ as in \cite[\S 6.5.4]{LM}, where $d= \dim X - \dim Y$, maps $\mathcal{N}_* (Y)$ into $\mathcal{N}_* (X)$.
In particular, it induces a map $f^*:\Omega_*^\num(Y)\to\Omega_{*+d}^\num(X)$.
\end{lemma}

\begin{proof}By Lemma \ref{stpdab4}, we have smooth compactifications $j_1 : X \hookrightarrow \widetilde{X}$ and $j_2: Y \hookrightarrow \widetilde{Y}$, and a morphism $\widetilde{f}: \widetilde{X} \to \widetilde{Y}$ such that $\widetilde{f} \circ j_1 = j_2 \circ f$. So, $f^* j_2 ^* = j_1 ^* \widetilde{f}^*$. Thus, $f^* \mathcal{N}_* (Y) = ^{\dagger} f^* j_2 ^* \mathcal{N}_* (\widetilde{Y}) = j_1 ^* \widetilde{f}^* \mathcal{N}_* (\widetilde{Y}) \subset j_1 ^* \mathcal{N}_* (\widetilde{X}) = ^{\ddagger} \mathcal{N}_* (X)$, where $\dagger$ and $\ddagger$ hold by Definition \ref{defn:smqp}, as desired.
\end{proof}

We strengthen Lemma \ref{stpdab1'}:

\begin{lemma}\label{lem:sm2pushfwd}Let $f: X \to Y$ be a projective morphism of smooth quasi-projective varieties. Then, the projective push-forward $f_*: \Omega_* (X) \to \Omega_* (Y)$ maps $\mathcal{N}_* (X)$ into $\mathcal{N}_* (Y)$. 
\end{lemma}

\begin{proof}By Lemma \ref{stpdab4}, we have smooth compactifications $j_1 : X \hookrightarrow \widetilde{X}$ and $j_2 : Y \hookrightarrow \widetilde{Y}$ and a morphism $\widetilde{f}: \widetilde{X} \to \widetilde{Y}$, forming a Cartesian square $\widetilde{f} \circ j_1 = j_2 \circ f$, with Tor-independent $j_2$ and $\widetilde{f}$. In particular, by \cite[Theorem 6.5.12]{LM}, $f_* j_1 ^* = j_2 ^* \widetilde{f}_*$. Then, by Definition \ref{defn:smqp}, $\mathcal{N}_* (X) = j_1 ^* \mathcal{N}_* (\widetilde{X})$.

Let $\alpha \in \mathcal{N}_* (X)$ so that $\alpha = j_1 ^* \beta$ for some $\beta \in \mathcal{N}_* (\widetilde{X})$. Then, by Lemma \ref{stpdab1'}, $\widetilde{f}_* \beta \in \mathcal{N}_* (\widetilde{Y})$. Hence, $j_2 ^* \widetilde{f}_* \beta \in j_2 ^* \mathcal{N}_* (\widetilde{Y}) = \mathcal{N}_* (Y)$. Since $f_* j_1 ^* \beta = j_2 ^* \widetilde{f}_* \beta$, we thus have $f_* \alpha \in \mathcal{N}_* (Y)$. That is, $f_* \mathcal{N}_* (X) \subset \mathcal{N}_* (Y)$. 
\end{proof}

\subsubsection{Quasi-projective schemes}Let $U$ be a quasi-projective $k$-variety. Let $f: U' \to U$ be a desingularization. By Lemma \ref{lem:smqp}, $\mathcal{N}_* (U')$ is well-defined. 

\begin{lemma}\label{stpdab2}If $f_{\ell}: U'_{\ell} \to U$, $\ell= 1, 2$ are two desingularizations of $U$, then $f_{1*} \mathcal{N}_* (U'_1) = f_{2*} \mathcal{N}_* (U'_2)$. 
\end{lemma}

\begin{proof}By Lemma \ref{stpdab3}, there is another desingularization $f_3: U'_3 \to U$ with $g_{\ell}: U'_3 \to U'_\ell$ for $\ell = 1,2$ such that $f_\ell \circ g_\ell = f_3$. So, by replacing $f_2$ by $f_3$, we may assume there is a morphism $g: U'_2 \to U'_1$ such that $f_1 \circ g = f_2$. 

As in Lemma \ref{lem:smqp}, choose two smooth compactifications $j_\ell: U'_{\ell} \hookrightarrow \widetilde{U}_{\ell}$ for $\ell = 1,2$. Using again the argument of Lemma \ref{stpdab3}, we may assume there exists a morphism $\widetilde{g}: \widetilde{U}_2 \to \widetilde{U}_1$ such that $\widetilde{g} \circ j_2 = j_1 \circ g$, so we have a commutative diagram
$$
\xymatrix{ & U_2 ' \ar[dl] _{f_2} \ar[d] ^g \ar@{^{(}->}[r] ^{j_2} &\widetilde{U}_2 \ar[d] ^{\widetilde{g}} \\
U &U_1 ' \ar[l] _{f_1} \ar@{^{(}->}[r] ^{j_1} & \widetilde{U}_1 }$$
and it gives $j_1 ^* \widetilde{g}_*= g_* j_2 ^*$ by \cite[Theorem 6.5.12]{LM}.

Let $\beta \in \mathcal{N}_* (U_2 ')$. By definition, $\mathcal{N}_* (U_2 ') = j_2 ^* \mathcal{N}_* (\widetilde{U}_2)$ so that $\beta = j_2 ^* \alpha$ for some $\alpha \in \mathcal{N}_* (\widetilde{U}_2)$. So, $f_{2*} \beta = f_{1*} g_* \beta = f_{1*} g_* j_2 ^* \alpha = f_{1*} j_1 ^* \widetilde{g}_* \alpha$. But, by Lemma \ref{stpdab1'}, $\widetilde{g}_* \alpha \in \mathcal{N}_* (\widetilde{U}_1)$, so that $j_1 ^* \widetilde{g}_* \alpha \in j_1 ^* \mathcal{N}_* (\widetilde{U}_1) =^{\dagger} \mathcal{N}_*( {U}_1')$, where $\dagger$ holds by Lemma \ref{lem:smqp}. Thus, $f_{2*} \beta \in f_{1*} \mathcal{N}_* (U_1 ')$, i.e. $f_{2*} \mathcal{N}_* (U_2') \subset f_{1*} \mathcal{N}_* (U_1')$.

To prove the other direction, first note that by \cite[Corollary 4.4.8.(2)]{LM}, $g_* (1_{U_2'}) = 1_{U'_1} + \sum_{i } \omega_i \beta_i$ for some $\omega_i \in \mathbb{L}_{>0}$ and $\beta_i \in \Omega^{\geq 1} (U_1')$. Since $(\sum_i \omega_i \beta_i)^{n+1}= 0$ when $n \geq \dim U_1'$, $\sum_i \omega_i \beta_i$ is nilpotent, so that $g_* (1_{U_2'})$ is a unit in $\Omega_* (U_1')$. Let $\gamma \in \Omega_* (U_1')$ be its inverse. Then, for any $\alpha \in \mathcal{N}_* (U_1')$, we have $f_{1*} \alpha = f_{1*} ( \alpha \cdot \gamma \cdot g_* (1_{U_2'})) = ^{\dagger} f_{1*} g_* (g^* (\alpha \cdot \gamma) \cdot 1_{U_2'}) = f_{2*} g^* (\alpha \cdot \gamma) \in ^{\ddagger} f_{2*} \mathcal{N}_* (U_2')$, where $\dagger$ holds by the projection formula, and $\ddagger$ holds because $\alpha \cdot \gamma \in \mathcal{N}_* (U_1')$ by Lemma \ref{lem:ideal_smqp} and $g^* (\alpha \cdot \gamma) \in \mathcal{N}_* (U_2')$ by Lemma \ref{lem:sm2pullbk}. This shows $f_{1*} \mathcal{N}_* (U_1') \subset f_{2*} \mathcal{N}_* (U_2')$. This completes the proof.
\end{proof}

\begin{definition}\label{defn:qp}
For any quasi-projective $k$-variety $U$, let $\mathcal{N}_* (U) := f_* \mathcal{N}_* (U')$ for any desingularization $f: U' \to U$ and $\mathcal{N}_* (U')$ is as in Definition \ref{defn:smqp}, and define $\Omega_* ^{\num} (U) := \Omega_* (U) / \mathcal{N}_* (U)$. By Lemma \ref{stpdab2}, $\mathcal{N}_* (U)$ is well-defined, hence so is $\Omega_* ^{\num} (U)$.

For $U$ in $\Sch_k$, let $U_i $ be the irreducible components of $U_{\red}$. Via the natural push-forward epimorphism $\bigoplus_i \Omega_* (U_i) \to \Omega_* (U)$ given by the projective map $\coprod U_i \to U$, we define $\mathcal{N}_* (U)$ to be the push-forward (in this case the image) of $\bigoplus_i  \mathcal{N}_* (U_i)$ in $\Omega_* (U)$. Since each $\mathcal{N}_* (U_i)$ is well-defined, so is $\mathcal{N}_* (U)$. Define $\Omega_* ^{\num} (U) = \Omega_* (U) / \mathcal{N}_* (U)$.
\end{definition}

\subsection{Basic functoriality}\label{sec:functoriality}
Unlike the case on $\Sm_k$ (as in Lemma \ref{lem:sm2pullbk}), for $\Omega_* ^{\num}$ on $\Sch_k$, one cannot expect to have $f^*$ for any morphism $f$ in $\Sch_k$. However, we can first check that we have projective push-forwards and smooth pull-backs. We first note a basic result:

\begin{lemma}\label{stpdab5}Let $f: X \to Y$ be a projective morphism of quasi-projective $k$-varieties. Then, there exist desingularizations $f_1: \widetilde{X} \to X$ and $f_2: \widetilde{Y} \to Y$ and a projective morphism $\widetilde{f}: \widetilde{X} \to \widetilde{Y}$ such that $f_2 \circ \widetilde{f} = f \circ f_1$.
\end{lemma}

\begin{proof}First choose any desingularization $f_2: \widetilde{Y} \to Y$. It gives a rational map $f': X \dashrightarrow \widetilde{Y}$. By choosing a desingularization $f_1 : \widetilde{X}\to X$ that resolves the indeterminacy of $f'$, we obtain a morphism $\widetilde{f}: \widetilde{X} \to \widetilde{Y}$ such that $f_2 \circ \widetilde{f} = f \circ f_1$. Since $f \circ f_1$ is projective and $f_2$ is separated, we deduce that $\widetilde{f}$ is projective, as desired.
\end{proof}

We strengthen Lemmas \ref{stpdab1'} and \ref{lem:sm2pushfwd}:

\begin{lemma}\label{lem:projpushfwd}
Let $f: X \to Y$ be a projective morphism in $\Sch_k$. Then, for the projective push-forward $f_* : \Omega_* (X) \to \Omega_* (Y)$, we have  $f_* : \Omega_* ^{\num} (X) \to \Omega_* ^{\num} (Y)$.
\end{lemma}

\begin{proof}
We may assume $X$ and $Y$ are irreducible. It is enough to show that $f_* \mathcal{N}_* (X) \subset \mathcal{N}_* (Y)$. Let $f_1, f_2, \widetilde{f}$ be as in Lemma \ref{stpdab5}. In particular, $f_{2*} \widetilde{f}_* = f_* f_{1*}$. Then, $f_* \mathcal{N}_* (X) = ^{\dagger} f_* f_{1*} \mathcal{N}_* (\widetilde{X})=  f_{2*} \widetilde{f}_* \mathcal{N}_* (\widetilde{X}) \subset ^{\ddagger}f_{2*} \mathcal{N}_* (\widetilde{Y}) =^{\dagger} \mathcal{N}_* (Y)$, where $\dagger$ hold by Definition \ref{defn:qp} and $\ddagger$ holds by Lemma \ref{lem:sm2pushfwd}.
\end{proof}

Recall that an l.c.i. morphism $f: X \to Y$ is a morphism that factors into $X \overset{i}{\to} P \overset{g}{\to} Y$, where $i$ is a regular embedding and $g$ is smooth. We show that $f^*$ exists for an l.c.i. morphism $f$ on $\Omega_* ^{\num}$. For smooth morphisms, we have:

\begin{lemma}\label{lem:smpullbk} Let $f: X \to Y$ be a smooth morphism of quasi-projective $k$-varieties. Let $d= \dim X - \dim Y$. Then, the smooth pull-back $f^* : \Omega_*  (Y) \to \Omega_{*+d}  (X)$ as in \cite[\S 2.1.2-3]{LM} maps $\mathcal{N}_* (Y)$ into $\mathcal{N}_* (X)$.
\end{lemma}

\begin{proof}Choose a desingularization $f_1: \widetilde{Y} \to Y$ and form the following Cartesian square:
$$\xymatrix{
\widetilde{X}:=X \times_Y \widetilde{Y} \ar[r] ^{\ \ \ \ \ \ \ \ pr_2} \ar[d] ^{pr_1} & \widetilde{Y} \ar[d] ^{f_1} \\
X \ar[r] ^{f} & Y.}$$
Since $f$ is smooth, so is $pr_2$. Since $\pi_{\widetilde{Y}} : \widetilde{Y} \to \Spec (k)$ is smooth, via the composition $\pi_{\widetilde{Y}} \circ pr_2$, $\widetilde{X}$ is a smooth quasi-projective $k$-scheme. So, by Lemma \ref{lem:sm2pullbk}, we have $pr_2^* \mathcal{N}_* (\widetilde{Y}) \subset \mathcal{N}_* (\widetilde{X})$. Hence, from the Cartesian diagram, we have $f^* \mathcal{N}_*(Y) =^{\dagger} f^* f_{1*} \mathcal{N}_* (\widetilde{Y}) = pr_{1*} pr_2 ^* \mathcal{N} _* (\widetilde{Y}) \subset pr_{1*} \mathcal{N}_* (\widetilde{X}) = ^{\dagger} \mathcal{N}_* (X)$, where $\dagger$ hold by Definition \ref{defn:qp}.

\end{proof}

Now, we check that the pull-back exists on $\Omega_* ^{\num}$ for a regular embedding:

\begin{lemma}\label{lem:reimpb}
Let $i: X \to Z$ be a regular embedding of codimension $c$ in $\Sch_k$. Let $i^*: \Omega_* (Z) \to \Omega_{*-c} (X)$ be the l.c.i. pull-back in \cite[\S 6.5.4]{LM}. Then, $i^* \mathcal{N}_* (Z) \subset \mathcal{N}_* (X)$, so that we have $i^*: \Omega^{\num}_* (Z) \to \Omega^{\num}_{*-c} (X)$.
\end{lemma}

\begin{proof}Since $\Omega_*(X_{\red}) = \Omega_* (X)$ for all schemes $X \in \Sch_k$, we may assume both $X$ and $Z$ are reduced. 

\noindent \textbf{Step 1.} First, consider the case when $Z$ is smooth projective, and $X$ is a strict normal crossing divisor on $Z$. Let $\{ X_{\ell} \}_{\ell \in I}$ be the irreducible components of $X$. Recall that by Definition \ref{defn:qp}, $\mathcal{N}_* (X)$ is the image of $\bigoplus_{\ell \in I} \mathcal{N}_* (X_\ell)$ via the pushforward map $\bigoplus_{\ell \in I} \Omega_* (X_\ell) \to \Omega_* (X)$ associated to the projective map $\coprod_{\ell \in I} X_{\ell} \to X$. So, if $i_{\ell} ^* \mathcal{N}_* (Z) \subset \mathcal{N}_* (X_{\ell})$ for each restriction $i_{\ell}: X_\ell \hookrightarrow X \hookrightarrow Z$, then we deduce $i^* \mathcal{N}_* (Z) \subset \mathcal{N}_* (X)$. Thus, by replacing $X$ by $X_{\ell}$, we may assume $X$ is a smooth divisor. But, this case is just a special case of Lemma \ref{lem:sm2pullbk}. This proves the lemma for Step 1.

\noindent \textbf{Step 2.} Now, suppose $Z$ is smooth quasi-projective, and $X$ is a strict normal crossing divisor on $Z$. Choose a smooth compactification $j : Z \hookrightarrow Z'$ and let $i': X \hookrightarrow X'$ be the Zariski closure of $Z$. By replacing $Z'$ by a further sequence of blow-ups of $Z'$, if necessary, we may suppose $X'$ is also a strict normal crossing divisor on $Z'$. Let $\{ X_{\ell} \}_{\ell \in I}$ and $\{ X_{\ell} '\}_{\ell \in I'}$ be the irreducible components of $X$ and $X'$, respectively. So, we have Cartesian squares
$$\xymatrix{
\coprod_{\ell\in I } X_{\ell} \ar[d]^{j''} \ar[r] ^p & X \ar[d]^{j'}  \ar[r] ^i & Z \ar[d] ^j \\
\coprod_{\ell \in I'} X'_{\ell} \ar[r] ^{p'} & X '\ar[r] ^{i'} & Z' ,}$$
where $p$ and $p'$ are the obvious projective morphisms, $j' = j|_{X}$, and $j''$ is the induced open immersion. They are desingularizations of $X$ and $X'$, respectively. Here, by Step 1, we know ${i'}^* \mathcal{N}_* (Z') \subset \mathcal{N}_* (X')$. On the other hand, we have $i^* \mathcal{N}_* (Z) =^{\dagger} i^* j^* \mathcal{N}_* (Z') = {j'}^* {i'}^* \mathcal{N}_* (Z') \subset {j'}^* \mathcal{N}_* (X')=^{\ddagger} {j'} ^*p'_* \mathcal{N}_* (\coprod X'_{\ell})= ^3 p_* {j''} ^* \mathcal{N}_* (\coprod X'_{\ell}) =^{\dagger} p_* \mathcal{N}_* (\coprod X_{\ell}) =^{\ddagger}  \mathcal{N}_* (X)$, where $\dagger$ holds by Definition \ref{defn:smqp}, $\ddagger$ holds by Definition \ref{defn:qp}, and $=^3$ holds by \cite[Theorem 6.5.12]{LM}. This settles the lemma in this case.

\noindent \textbf{Step 3.} Now, consider the general case when $i: X \hookrightarrow Z$ is a regular embedding of codimension $d$, with $Z$ is a quasi-projective $k$-variety. Choose (via Hironaka) a sequence of blow-ups $p: \widetilde{Z} \to Z$ such that $\widetilde{Z}$ is smooth quasi-projective, and $\widetilde{X}:= p^{-1} (X)$ is a strict normal crossing divisor, giving a Cartesian diagram 
$$\xymatrix{
\widetilde{X} \ar[d] ^{q} \ar[r] ^{\widetilde{i}} & \widetilde{Z} \ar[d] ^p \\
X \ar[r] ^{i} & Z,}$$
where $q$ is the restriction of $p$. Here, $i^* \mathcal{N}_* (Z) = ^{\dagger} i^* p_* \mathcal{N}_* (\widetilde{Z})= ^{\ddagger} q_* \widetilde{i}^* \mathcal{N}_* (\widetilde{Z}) \subset ^3 q_* \mathcal{N}_* (\widetilde{X}) \subset^4 \mathcal{N}_* (X)$, where $\dagger$ holds by Definition \ref{defn:qp}, $\ddagger$ holds by \cite[Theorem 6.5.12]{LM}, $\subset^{3}$ holds by Step 2, and $\subset^{4}$ holds by Lemma \ref{lem:projpushfwd}. This proves the lemma in general.
\end{proof}

When $f:X \to Y$ is an l.c.i. morphism, one may have $f = p_1 \circ i_1 = p_2 \circ i_2$. Since $i_1 ^* \circ p_1 ^* = i_2 ^* \circ p_2 ^*$ on $\Omega_* (Y)$ by \cite[Lemma 6.5.9]{LM}, and the natural map $\Omega_* (-) \to \Omega_* ^{\num} (-)$ is surjective, the same equality holds on $\Omega_* ^{\num} (Y)$. Thus:

\begin{definition}For an l.c.i. morphism $f: X \to Y$, define $f^*: \Omega_* ^{\num} (Y) \to \Omega_{*+d} ^{\num} (X)$ to be $i^* \circ p^*$, where $d= \dim X - \dim Y$.
\end{definition}

Using the surjectivity of $\Omega_* (-) \to \Omega_* ^{\num} (-)$ and \cite[Theorems 6.5.11, 6.5.12, 6.5.13]{LM}, one deduces that $f^* \circ g^* = (g \circ f)^*$ and $(f_1 \times f_2)^* = f_1 ^* \times f_2 ^*$ for l.c.i. morphisms $f, g, f_1, f_2$. One also deduces that if $f: X \to Z$ and $g: Y \to Z$ are Tor-independent morphisms in $\Sch_k$, where $f$ is l.c.i. and $g$ is projective, then for the Cartesian product $X \times_Z Y$, we have $f^* \circ g_* = pr_{1*} \circ pr_2 ^*$, where $pr_1, pr_2$ are the projections from $X \times_Z Y$ to $X$ and $Y$, respectively.

\subsection{Chern class operation}We check that we have the first Chern class operators:

\begin{lemma}\label{lem:chern class}
The first Chern class operators $\widetilde{c}_1 (E)$ on $\Omega_*$ as in \cite[\S 2.1.2]{LM} induce well-defined operators $\widetilde{c}_1(E):\Omega_*^\num(X)\to\Omega_{*-1}^\num(X)$ for vector bundles $E\to X$.
\end{lemma}

\begin{proof}Let $\alpha \in \mathcal{N}_* (X)$. We show that $\widetilde{c}_1 (E) (\alpha)  \in \mathcal{N}_* (X)$. By Definition \ref{defn:qp}, $\mathcal{N}_* (X) = f_* \mathcal{N}_* (\widetilde{X})$ for a desingularization $f: \widetilde{X} \to X$, so there is $\beta \in \mathcal{N}_* (\widetilde{X})$ such that $\alpha = f_* \beta$. By the projection formula, we have $\widetilde{c}_1 (E) (\alpha) = \widetilde{c}_1 (E) (f_* \beta) = f_* ( \widetilde{c}_1 (f^* (E))(\beta))$. But, since $\widetilde{X}$ is smooth, by \cite[(5.2)-5]{LM} we have $\widetilde{c}_1 (f^* (E)) (\beta) = c_1 (f^* (E)) \cdot \beta$, which lies in $\mathcal{N}_* (\widetilde{X})$ by Lemma \ref{lem:ideal_smqp}. Hence, $\widetilde{c}_1 (E) (\alpha) = f_* (c_1 (f^* (E)) \cdot \beta) \in f_* \mathcal{N}_* (\widetilde{X}) = \mathcal{N}_* (X)$, as desired.
\end{proof}

\subsection{Localization Sequence}Observe the following simple fact:

\begin{lemma}\label{stpdab6}Let $j: U \hookrightarrow X$ be an open inclusion between quasi-projective $k$-varieties. Then, for the map $j^* : \Omega _* (X) \to \Omega_* (U)$, the map $j^* : \mathcal{N}_* (X) \to \mathcal{N}_* (U)$ as in Lemma \ref{lem:smpullbk} is surjective. Furthermore, $(j^*)^{-1} (\mathcal{N}_* (U)) = \mathcal{N}_* (X)$. 
\end{lemma}

\begin{proof}Let $f: \widetilde{X} \to X$ be a desingularization of $X$. By Definition \ref{defn:qp}, $\mathcal{N}_* (X) = f_* \mathcal{N}_* (\widetilde{X})$. Consider the Cartesian diagram
$$\xymatrix{
\widetilde{U}:= U \times_X \widetilde{X} \ar@{^{(}->}[r] ^{\ \ \ \ \ \ \ \ \widetilde{j}} \ar[d] ^{f_U} & \widetilde{X} \ar[d] ^{f} \\
U \ar@{^{(}->}[r] ^j & X,}$$ where $\widetilde{j}$ is the induced open immersion. From this, we deduce $j^* f_* = f_{U*} \widetilde{j}^*$. Since $f_U$ is a desingularization of $U$, by Definition \ref{defn:qp}, we have $\mathcal{N}_* (U) = f_{U*} \mathcal{N}_* (\widetilde{U})$. Let $\alpha \in \mathcal{N}_* (U)$. Write $\alpha = f_{U_*} \beta$ for some $\beta \in \mathcal{N}_* (\widetilde{U})$. Since $\widetilde{j}^* : \Omega_* (\widetilde{X}) \to \Omega_* (\widetilde{U})$ is a surjective ring homomorphism by \cite[Theorem 3.2.7]{LM}, and since $\mathcal{N}_* (\widetilde{U}) = \widetilde{j}^* \mathcal{N}_* (\widetilde{X})$ by definition, which is an ideal by Lemma \ref{lem:ideal_smqp}, by the correspondence theorem of ideals, we have $(\widetilde{j}^*)^{-1} \mathcal{N}_* (\widetilde{U}) = \mathcal{N}_* (\widetilde{X})$. In particular, there is $\gamma \in \mathcal{N}_* (\widetilde{X})$ such that $\widetilde{j}^* \gamma = \beta$. So, $j^* f_* \gamma = f_{U*} \widetilde{j}^* \gamma = f_{U*} \beta = \alpha$. But, $f_* \gamma \in f_* \mathcal{N}_* (\widetilde{X}) = \mathcal{N}_* (X)$ by definition, so, $j^*: \mathcal{N}_* (X) \to \mathcal{N}_* (U)$ is surjective. The second assertion follows from the above argument.
\end{proof}

\begin{theorem}\label{thm:localization}
Let $X$ be a quasi-projective variety, $i:Z\to X$ a closed subscheme and $j:U\to X$ the open complement. Then, the following sequence is exact:
\[ \Omega_*^\num(Z)\overset{i_*}{\to}\Omega_*^\num(X)\overset{j^*}{\to}\Omega_*^\num(U)\longrightarrow 0.\]
\end{theorem}

\begin{proof}
Consider the following commutative diagram
\begin{equation} \label{eq:loc}
\xymatrix{ 0 \ar[d] & 0 \ar[d] & 0 \ar[d] & \\
\sN_*(Z) \ar[d]^{h_Z} \ar[r]^{i_* ^{\mathcal{N}}} & \sN_*(X)  \ar[d]^{h_X} \ar[r]^{j^*_{\mathcal{N}}} & \sN_*(U) \ar[d]^{h_U} \ar[r] & 0 \\
\Omega_*(Z) \ar[d]^{g_Z} \ar[r]^{i_*} & \Omega_*(X) \ar[d]^{g_X} \ar[r]^{j^*} & \Omega_*(U) \ar[d]^{g_U} \ar[r] & 0 \\
\Omega_*^\num(Z) \ar[d] \ar[r]^{i_* ^{\num}} & \Omega_*^\num(X) \ar[d] \ar[r]^{j^*_{\num}} & \Omega_*^\num(U) \ar[d] \ar[r] & 0 \\
0 & 0 & 0, & }
\end{equation}
where the columns are exact by definition, the middle row is exact by \cite[Theorem 3.2.7]{LM}, $i_*^{\num}$ and $j^*_{\num}$ are as in Lemmas \ref{lem:projpushfwd}, \ref{lem:smpullbk}. Clearly, $j_{\num} ^* \circ i_* ^{\num} = 0$, $j^* \circ i_* = 0$, $j^*_{\num} \circ i_* ^{\num}=0$ and $j^*_{\num}$ is surjective.  The map $j^*_{\mathcal{N}}$ is surjective by Lemma \ref{stpdab6}.

Since ${\rm im}(i_* ^{\num}) \subset \ker (j^*_{\num})$, it remains to show ${\rm im}(i_* ^{\num}) \supset \ker (j^* _{\num})$. Given $\bar{\alpha} \in \ker (j^* _{\num})$, choose $\alpha \in \Omega_* (X)$ such that $g_X (\alpha) = \bar{\alpha}$. Then, $g_U (j^* (\alpha)) = j_{\num}^* (g_X (\alpha)) = j_{\num} ^* (\ov{\alpha}) = 0$, so that $j^* (\alpha) \in \mathcal{N}_* (U)$. By Lemma \ref{stpdab6}, there is $\beta \in \mathcal{N}_* (X)$ such that $j_{\mathcal{N}}^* (\beta) = j^* (\alpha)$. Note $g_X (\beta) =0$. Then, $j^* (\alpha - \beta)  = 0$ so that by the exactness of the middle row, $\alpha - \beta = i_* (\gamma)$ for some $\gamma \in \Omega_* (Z)$. Hence, $\bar{\alpha} = g_X (\alpha) = g_X (\beta + i_* (\gamma)) = g_X (\beta ) + g_X i_* (\gamma) = i_{*} ^{\num} (g_Z (\gamma))$. Thus, $\bar{\alpha} \in {\rm im} (i_* ^{\num})$ as desired.
\end{proof}

\subsection{Homotopy invariance}\label{sec:homo inv}

\begin{theorem}\label{thm:EH} Let $X$ be a quasi-projective $k$-scheme of finite type, and let $f: E \to X$ be a torsor under a vector bundle over $X$ of rank $n$. Then, the map $f^*: \Omega_* ^{\num} (X) \to \Omega_{*+n} ^{\num} (E)$ is an isomorphism.
\end{theorem}

\begin{proof}Since $\Omega_* (X) = \Omega_* (X_{\rm red})$, we may assume $X$ is reduced. For surjectivity, consider the commutative diagram
$$
\begin{CD}
\Omega_* (X) @>{f^*}>> \Omega_{*+n} (E) \\
@VVV @VVV \\
\Omega_* ^{\num} (X) @>{f^*}>> \Omega_{*+n} ^{\num} (E).
\end{CD}$$By the homotopy invariance of $\Omega_*$ in \cite[Theorem 3.4.2]{LM}, the top horizontal arrow is an isomorphism. On the other hand, by construction, the two vertical arrows are surjective. Hence, the bottom arrow $f^*$ is surjective. 

For injectivity (which implies then that the bottom $f^*$ is an isomorphism), we argue in steps. First consider the case when $f: E \to X$ is a trivial bundle. Then, by induction on $n$, we reduce to the case when $n=1$, i.e. $E= X \times \mathbb{A}^1$. Consider the zero section map $i_{0}^X : X \to X \times \mathbb{A}^1$ so that $f \circ i_{0}^X = {\rm Id}_X$. Here, $i_0: \{ 0 \} \hookrightarrow \mathbb{A}^1$ is a regular embedding, and the maps $pr_2: X \times \mathbb{A}^1 \to \mathbb{A}^1$ and $i_0$ are Tor-independent by \cite[Lemma 6.3]{KP} so that $i_0 ^X : X = \{0 \} \times_{\mathbb{A}^1} (X \times \mathbb{A}^1) \to X \times \mathbb{A}^1$ is also a regular embedding. In particular $(i_0 ^{X})^*$ exists on $\Omega_* ^{\num}$ by Lemma \ref{lem:reimpb}, and it satisfies $(i_0 ^X)^* \circ f^* = {\rm Id}_X ^*$, showing that $f^*: \Omega_{*} ^{\num}(X) \to \Omega_{*+n} ^{\num} (E) $ must be injective. Thus, we just proved that $f^*: \Omega_{*} ^{\num}(X) \to \Omega_{*+n} ^{\num} (E) $ is an isomorphism for all trivial bundles. 

In general, we use the standard Noetherian induction argument on $\dim X$. If $\dim X = 0$, then $X$ is just a finite set of points so that $f: E \to X$ is a trivial bundle. Thus, $f^*: \Omega_{*} ^{\num}(X) \to \Omega_{*+n} ^{\num} (E) $ is an isomorphism as seen before. Suppose the theorem holds for schemes of dimension $< \dim X$. Let $U \subset X$ be a nonempty Zariski open subscheme where $f_U: E|_U \to U$ is a trivial bundle. By the case of the trivial bundle treated before, $f^*_U : \Omega_{*} ^{\num}(U) \to \Omega_{*+n} ^{\num} (E|_U) $ is an isomorphism. Let $Y:= ( X \setminus U)_{\rm red}$. By the induction hypothesis, $f_Y: E|_Y \to Y$ is induces an isomorphism $f_Y ^* : \Omega_* ^{\num} (Y) \to \Omega_{*+n} ^{\num} (E|_Y)$. So, we have a commutative diagram whose rows are localization sequences of Theorem \ref{thm:localization}
$$\xymatrix{\Omega_*^\num (Y) \ar[r] \ar[d]^{f_Y ^*} & \Omega_*^\num (X) \ar[r] \ar[d]^{f^*} & \Omega_*^\num (U) \ar[r] \ar[d]^{f_U^*} & 0\\
\Omega_{*+n}^\num (E|_Y) \ar[r]  & \Omega_{*+n}^\num (E) \ar[r]  & \Omega_{*+n}^\num (E|_U) \ar[r]  & 0,}$$
where the first and the third vertical arrows are isomorphisms. Thus, by a simple diagram chasing, the middle vertical arrow is an isomorphism.
\end{proof}

\subsection{Projective bundle formula}\label{sec:pbf}

Let $X$ be a $k$-scheme of finite type. Let $p:\sE\to X$ be a vector bundle of rank $n+1$ with projectivization $q:\P(\sE)\to X$. Let $\xi:=\widetilde{c}_1(\sO(1))$. To establish the Projective bundle formula on $\Omega_* ^{\num}$, we follow some ideas from \cite[\S 3.5]{LM}. Recall from \cite[\S 3.5.2]{LM}, the operator $\Phi_{X,\sE}:\oplus_{j=0}^n\Omega_{*-n+j}(X)\to \Omega_*(\P(\sE))$ given as the sum of $\phi_j:=\xi^j\circ q^*:\Omega_{*-n+j} (X)\to \Omega_*(\P(\sE))$.
By Lemmas \ref{lem:smpullbk} and \ref{lem:chern class}, $\Phi_{X,\sE}$ descends to a well-defined operator $\bar{\Phi}_{X,\sE}:\oplus_{j=0}^n\Omega_{*-n+j}^\num(X)\to \Omega_*^\num(\P(\sE)).$ We assert that this is an isomorphism. 

\begin{lemma}[{\emph{cf.} \cite[Lemma~3.5.1]{LM}}]\label{lem:closed}
Let $i:F\to X$ be a closed subset and let $\Omega^{\num}_{*,F}(X)$ be the image of $i_*$ in $\Omega^{\num}_*(X)$. Then, we have
\begin{enumerate}
\item For a line bundle $L\to X$, $\Omega^{\num}_{*,F}(X)$ is stable under the operator $\widetilde{c}_1(L)$.
\item Let $p:Y\to X$ be a smooth quasi-projective morphism and $q:X\to Z$ be a projective morphism. Then, $p^*\Omega^{\num}_{*,F}(X)\subset \Omega^{\num}_{*,p^{-1}(F)}(Y)$ and $q_*\Omega^{\num}_{*,F}(X)\subset \Omega^{\num}_{*,q(F)}(Z)$.
\end{enumerate}
\end{lemma}

\begin{proof}
It follows from the results of \S\ref{subsec:definition} and \ref{sec:functoriality} that $\Omega^{\num}_*$ is an oriented Borel-Moore functor in the sense of \cite[Definition~2.1.2]{LM}), whence the same proof as that of \cite[Lemma~3.5.1]{LM} works.
\end{proof}

\begin{theorem}\label{thm:PB}
In the above notation, $\bar{\Phi}_{X,\sE}$ is an isomorphism.
\end{theorem}

\begin{proof}
We first consider the case of the trivial bundle, so that $\P(\sE)=\P^n_X$. We already know $\mathcal{N}_* (\P(\sE))\supset \Phi_{X,\sE}(\oplus_{j=0}^n\sN_{*-n+j}(X))$. Since $\Phi_{X,\sE}$ is an isomorphism by \cite[Theorem~3.5.4]{LM}, it is enough to show that $\sN_*(\P(\sE))\subset \Phi_{X,\sE}(\oplus_{j=0}^n\sN_{*-n+j}(X))$. Note that the operator $\Psi_{X,\sE}$ defined in \cite[\S3.5.3]{LM} descends to a well-defined operator $\bar{\Psi}_{X,\sE}:\Omega_*^\num(\P(\sE))\to \oplus_{j=0}^n\Omega_{*-n+j}^\num(X)$ with $\bar{\Psi}_{X,\sE}\circ\bar{\Phi}_{X,\sE}={\rm Id}$. This implies that $\bar{\Phi}_{X,\sE}$ is injective.
Consider the localization sequence given by the inclusion $i:\P^{n-1}_X\to\P^n_X$ with open complement $j:\A^n _X \to\P^n_X $
\begin{equation}\label{eq:PBloc}
\Omega_*(\P^{n-1}_X)\overset{i_*}{\to} \Omega_*(\P^n_X)\overset{j^*}{\to}\Omega_*(\A^n_X)\longrightarrow 0.
\end{equation}
It is shown in the course of the proof of \cite[Theorem~3.5.4]{LM} that $j^*\circ \Phi_{X,\sE}=p^*$. Thus, by Theorem \ref{thm:EH}, 
$j^*\sN_*(\P(\sE))=\sN_*(\A^n _X)=p^*\sN_*(X)=j^*\Phi_{X,\sE}\sN_*(X)$. Then, for any $\alpha\in \sN_*(\P(\sE))$, there is a $\beta\in\sN_*(X)$ such that $\alpha-\Phi_{X,\sE}(\beta)\in\ker (j^*)={\rm im} (i_*)$. It follows from \textit{loc.\ cit.}\ that the image of $i_*$ is contained in $\Phi_{X,\sE}(\oplus_{j=0}^n\Omega_{*-n+j}(X))$. Thus, $\alpha-\Phi_{X,\sE}(\beta)=\Phi_{X,\sE}(\gamma)$ for some $\gamma\in\oplus_{j=0}^n\Omega_{*-n+j}(X)$. Since $\bar{\Phi}_{X,\sE}$ is injective, $\Phi_{X,\sE}(\gamma)\in\sN_*(\P(\sE))\Rightarrow \gamma\in\oplus_{j=0}^n\sN_{*-n+j}(X)$, whence $\alpha\in \Phi_{X,\sE}(\oplus_{j=0}^n\sN_{*-n+j}(X))$. Thus, $\sN_*(\P(\sE))\subset \Phi_{X,\sE}(\oplus_{j=0}^n\sN_{*-n+j}(X))$, as desired.

Now suppose $E \to X$ is a general vector bundle. The idea is similar as in \textit{loc.\ cit.} Choose a descending filtration $\{X_m\}_{m=0}^N$ of $X$ by closed subschemes such that the restriction of $\sE$ to $Y_m:= X_m\setminus X_{m+1}$ is trivial. Thus, $\bar{\Phi}_{Y_m,\sE_{Y_m}}$ is an isomorphism with inverse $\bar{\Psi}_{Y_m,\sE_{Y_m}}$. By Lemma~\ref{lem:closed}, they descend to maps 
\begin{align*}
\bar{\Phi}^{Y_m} : & \oplus_{j=0}^n\Omega^{\num}_{*-n+j, Y_m}(X\setminus X_{m+1})\to \Omega^{\num}_{*,\P^n_{Y_m}}(\P_{X\setminus X_{m+1}}(\sE)) \\
\bar{\Psi}^{Y_m} : & \;\Omega^{\num}_{*,\P^n_{Y_m}}(\P_{X\setminus X_{m+1}}(\sE)) \to \oplus_{j=0}^n\Omega^{\num}_{*-n+j,Y_m}(X\setminus X_{m+1})
\end{align*}
with $\bar{\Psi}^{Y_m}\circ\bar{\Phi}^{Y_m}={\rm Id}$. This implies that $\bar{\Phi}^{Y_m}$ is an isomorphism. Assume by induction that $\bar{\Phi}_{X\setminus X_m}$ is an isomorphism. Note that, to ease the notation, we remove the appropriate restriction of $\sE$ from the subscript of $\bar{\Phi}$. By Theorem~\ref{thm:localization}, we now have the localization sequences for the inclusion $Y_m\overset{i}{\to} X\setminus X_m$ giving us the commutative diagram with exact columns

\[ \xymatrix{
0 \ar[d] & 0 \ar[d] \\
\oplus_{j=0}^n\Omega^{\num}_{*-n+j,Y_m}(X\setminus X_{m+1}) \ar[r]^{ \ \ \ \bar{\Phi}^{Y_m}} \ar[d] & \Omega^{\num}_{*,\P^n_{Y_m}}(\P_{X\setminus X_{m+1}}(\sE)) \ar[d] \\
\oplus_{j=0}^n\Omega^{\num}_{*-n+j}(X\setminus X_{m+1}) \ar[r]^{\ \ \bar{\Phi}_{X\setminus X_{m+1}}} \ar[d]^{j^*} & \Omega^{\num}_*(\P_{X\setminus X_{m+1}}(\sE)) \ar[d]^{j^*} \\
\oplus_{j=0}^n\Omega^{\num}_{*-n+j}(X\setminus X_{m}) \ar[r]^{\ \ \bar{\Phi}_{X\setminus X_{m}}} \ar[d] & \Omega^{\num}_*(\P_{X\setminus X_{m}}(\sE)) \ar[d] \\
0 & 0.} \]
Since $\bar{\Phi}^{Y_m}$ and $\bar{\Phi}_{X\setminus X_m}$ are isomorphisms, so is $\bar{\Phi}_{X\setminus X_{m+1}}$ and we have the desired result by induction.
\end{proof}

Combining all the above results in \S \ref{sec:OBM}, we deduce finally:

\begin{theorem}\label{thm:OBM}The functor $\Omega_* ^{\num}$ is an oriented Borel-Moore homology theory on $\Sch_k$, and its restriction $\Omega_{\num} ^*$ on $\Sm_k$ is an oriented cohomology theory.
\end{theorem}

\subsection{Generic constancy} Recall (\cite[Definition 4.4.1]{LM}) that we say an oriented Borel-Moore homology theory $A_*$ on $\Sch_k$ is \emph{generically constant} if for each finitely generated field extension $k \subset F$, the natural morphism $A_* (k) \to A_* (F/k)$ is an isomorphism, where $A_* (F/k) = \colim A_{*+{\rm tr}_{F/k}} (X)$ in which the colimit is taken over all models $X$ for $F$ over $k$. Recall that a model for $F$ over $k$ is an integral scheme $X$ with $k(X) \simeq F$. We have:

\begin{proposition}\label{prop:generic constant}The cobordism theory $\Omega_* ^{\num}$ is generically constant.
\end{proposition}

\begin{proof}We have a commutative diagram
$$\xymatrix{ \Omega_* (k) \ar[d] ^{\eta_F} _{\simeq} \ar[r] ^{\simeq} & \Omega_* ^{\num} (k) \ar[d] ^{\eta_F ^{\num}} \ar[dr] ^{\simeq} & \\
\colim _{\mathcal{C}} \Omega_{*+{\rm tr}_{F/k} (X)} \ar[r] & \colim_{\mathcal{C}} \Omega_{*+ {\rm tr}_{F/k}} ^{\num} (X) \ar[r] & \Omega_* ^{\num} (F),}$$
where $\mathcal{C}$ is the category of models for $F$ over $k$. Here, the top isomorphism and the right sloped isomorphism hold by Remark \ref{rmk:point}, and the left vertical arrow is an isomorphism by \cite[Corollary 4.4.3]{LM}. The first bottom horizontal arrow is a natural surjection, so, $\eta_F ^{\num}$ is surjective. But, the right commutative triangle shows $\eta_F ^{\num}$ is injective, too. Thus, $\eta_F ^{\num}$ is an isomorphism.
\end{proof}

\begin{corollary}\label{cor:OBMgc}The theory $\Omega_* ^{\num}$ on $\Sch_k$ is a generically constant oriented Borel-Moore homology theory on $\Sch_k$.
\end{corollary}
\begin{proof}The assertion follows from Theorem \ref{thm:OBM} and Proposition \ref{prop:generic constant}.
\end{proof}

\begin{remark}We remark that $\Omega _* ^{\num}$ satisfies the generalized degree formula by \cite[Theorem 4.4.7]{LM}, for it holds for any oriented Borel-Moore weak homology theory, that is generically constant and has the localization property. All these are implied by Corollary \ref{cor:OBMgc}.
\end{remark}

\subsection{Comparison with $\Omega_* ^{\alg}$} Recall from \cite{KP} the theory $\Omega_* ^{\alg}$ of algebraic cobordism modulo algebraic equivalence. We have a natural surjective morphism of oriented Borel-Moore homology theories $\Omega_* \to \Omega_* ^{\alg}$. We now show that there is a natural surjection $\Omega_* ^{\alg} \to \Omega_* ^{\num}$, and it is a morphism of oriented Borel-Moore homology theories, too. By \cite[Theorem 6.4]{KP}, for each $X \in \Sch_k$, there is an exact sequence
\begin{equation}\label{eqn:stpdabalg}
\bigoplus \Omega_* (X \times C) \overset{i_1 ^* - i_2 ^*}{\to} \Omega_* (X) \to \Omega_* ^{\alg} (X) \to 0,
\end{equation}
where the sum runs over the equivalence classes $(C, t_1, t_2)$ of triples consisting of a smooth projective connected curve $C$ and two distinct points $t_1, t_2 \in C(k)$, and $i_j ^*$ is the pull-back via the closed immersion $i_j : X \times \{ t_j \} \to X \times C$. Since $\Omega_* (X) \to \Omega_* ^{\num} (X)$ is surjective by definition, it is enough to show: 

\begin{proposition}For each $X \in \Sch_k$, we have $\iota (\mathcal{A}(X)) \subset \mathcal{N}_* (X)$, where $\mathcal{A}(X):= \bigoplus \Omega_* (X \times C)$ and $\iota$ is the sum of $i_1 ^* - i_2 ^*$.
\end{proposition}

\begin{proof}We may assume $X$ is integral.

\textbf{Step 1.} First suppose $X$ is a smooth projective $k$-variety. Let $\pi: X \to \Spec (k)$ be the structure morphism. Observe that we have the following, whose proof is easy:

\begin{lemma}\label{lem:eastpdabzy}
Let $X$ be a smooth quasi-projective variety. Then,
\begin{enumerate}
\item  For $\alpha \in \mathcal{A}(X)$ and $\beta \in \Omega_* (X)$, we have $\alpha \times \beta, \beta \times \alpha \in \mathcal{A} (X \times X)$.
\item  For $\alpha \in \mathcal{A}(X \times X)$, we have $\Delta_X ^* \alpha \in \mathcal{A} (X)$.
\item $\iota (\mathcal{A}(X))$ is an ideal of the ring $\Omega_* (X)$.
\end{enumerate}
\end{lemma}

Continuing the proof, by \eqref{eqn:stpdabalg} and Lemma \ref{lem:eastpdabzy}, we have a commutative diagram, where the columns are exact and $\mathcal{B} (X) := \mathcal{A} (X) \otimes \Omega_* (X)\oplus \Omega_* (X) \otimes \mathcal{A}(X)$:
$$
\xymatrix{\mathcal{B}(X) \ar[d]^{\iota \otimes {\rm Id} + {\rm Id} \otimes \iota} \ar[r] ^{\times} & \mathcal{A} (X \times X) \ar[d]^{\iota} \ar[r] ^{\Delta_X ^*} & \mathcal{A}(X) \ar[d]^{\iota} \ar[r] ^{\pi_*} & \mathcal{A}(k)  \ar[d]^{\iota} \\
\Omega_* (X) \otimes \Omega_* (X) \ar[d] \ar[r] ^{\times} & \Omega_* (X \times X) \ar[d]  \ar[r] ^{\Delta_X ^*} & \Omega_* (X) \ar[d]^{\Phi} \ar[r] ^{\pi_*} & \Omega_* (k) \ar[d]^{\simeq} \\
\Omega_* ^{\alg} (X) \otimes \Omega_* ^{\alg} (X) \ar[d] \ar[r] ^{\times} & \Omega_* ^{\alg} (X \times X) \ar[d]  \ar[r] ^{\Delta_X ^*} & \Omega_* ^{\alg} (X) \ar[d] \ar[r] ^{\pi_*} & \Omega_* ^{\alg} (k) \ar[d] \\
0 & 0 & 0 & 0,}
$$where the far right middle isomorphism comes from \cite[Theorem 1.2(2)]{KP}. Thus, $\mathcal{A}(k) = 0$. Thus, if $\alpha \otimes \gamma \in \mathcal{A} (X) \otimes \Omega_* (X)$, then $\pi_* (\alpha \cdot \gamma) = 0$, so by the commutativity we get $\iota (\alpha) \in \mathcal{N}_* (X)$, which shows $\iota (\mathcal{A}(X)) \subset \mathcal{N}_* (X)$. This proves the proposition in this case.

\noindent \textbf{Step 2.} Now suppose $X$ is smooth quasi-projective. Choose a smooth compactification $j: X \hookrightarrow \widetilde{X}$. Then, we have a commutative diagram with exact rows and columns, where the rows come from \eqref{eqn:stpdabalg} and the columns are part of localization sequences:
$$\xymatrix{0 \ar[r] & \iota (\mathcal{A}(\widetilde{X})) \ar[d] ^{j^*_{\mathcal{A}}} \ar[r] & \Omega_* (\widetilde{X}) \ar[d] ^{j^*} \ar[r] & \Omega_* ^{\alg} (\widetilde{X}) \ar[d] ^{j^*_{\alg}} \ar[r] & 0 \\
0 \ar[r] & \iota(\mathcal{A}(X)) \ar[d]  \ar[r] & \Omega_* (X) \ar[d] \ar[r] & \Omega_* ^{\alg} (X) \ar[d] \ar[r] & 0 \\
 & 0 & 0 & 0,
}$$ Note that $j^*$ is a surjective ring homomorphism, while $\iota (\mathcal{A}(\widetilde{X})) \subset \Omega_* (\widetilde{X})$, $\iota (\mathcal{A} (X)) \subset \Omega_* (X)$ are ideals, so that by the commutativity and the correspondence theorem $j_{\mathcal{A}} ^*$ is also surjective, i.e. $j_{\mathcal{A}} ^* ( \iota (\mathcal{A} (\widetilde{X}))) = \iota (\mathcal{A} (X))$. By Step 1, we know $\iota (\mathcal{A} (\widetilde{X})) \subset \mathcal{N}_* (\widetilde{X})$, so, applying $j^*$ we obtain $\iota (\mathcal{A}(X)) = j_{\mathcal{A}} ^* (\iota (\mathcal{A}(\widetilde{X}))) \subset j^* \mathcal{N}_* (\widetilde{X}) = ^{\dagger} \mathcal{N}_* (X)$, where $\dagger$ holds by Definition \ref{defn:smqp}, proving the proposition in this case.

\noindent \textbf{Step 3.} Now suppose $X$ is a quasi-projective variety. We prove it by induction on $\dim X$. If $\dim X = 0$, or $X$ is smooth, it is already done by Step 2. So, suppose $\dim X >0$ and let $i: Z \hookrightarrow X$ be the singular locus of $X$. Here, $\dim Z < \dim X$. Suppose that the proposition holds for all schemes of dimension $< \dim X$. Choose a desingularization $f: \widetilde{X} \to X$. Consider the commutative diagram with exact rows
$$\xymatrix{
\mathcal{A} (\widetilde{X}) \ar[d] ^{f_* ^{\mathcal{A}}} \ar[r] ^{\iota_{\widetilde{X}}} & \Omega_* (\widetilde{X}) \ar[d] ^{f_*} \ar[r] & \Omega_* ^{\alg} (\widetilde{X}) \ar[d] ^{f_* ^{\alg}} \ar[r] & 0\\
\mathcal{A} (X) \ar[r] ^{\iota_X} & \Omega_* (X) \ar[r] & \Omega_* ^{\alg} (X) \ar[r] & 0,
}$$
where, $\iota$ is the sum of $i_1 ^* - i_2 ^*$, with $i_j : X \times \{ t_j \} \to X \times C$, regular embeddings, so that the left square commutes by \cite[Theorem 6.5.12]{LM}. Now, by the equation \cite[(5.5)]{KP}, it is proven that every class $\alpha \in \mathcal{A}(X)$ can be written as $\alpha = f_* ^{\mathcal{A}} (\widetilde{\alpha}) + i_* \xi$ for some cobordism cycles $\widetilde{\alpha} \in \mathcal{A} (\widetilde{X})$ and $\xi \in \mathcal{A} (Z)$. 

For the first term, we have $\iota_X f_* ^{\mathcal{A}} (\widetilde{\alpha}) = ^{\dagger} f_* \iota_{\widetilde{X}} (\widetilde{\alpha}) \in ^{\ddagger} f_* \mathcal{N}_* (\widetilde{X})  \subset ^3 \mathcal{N}_* (X)$, where $\dagger$ holds by the commutativity of the left square, $\ddagger$ holds by Step 2., and $\subset ^3$ holds by Definition \ref{defn:qp}. For the second term, since $\dim Z < \dim X$, by the induction hypothesis, we have $\iota_Z \mathcal{A} (Z) \subset \mathcal{N}_* (Z)$. But, we have $i_* \iota_Z = \iota_X i_*$ (which is proven in the line below the diagram in \cite[p.89]{KP} using \cite[Proposition 6.5.4]{LM}), so that $\iota_X i_* \eta  = i_* \iota_Z \eta \in i_*  \mathcal{N}_* (Z) \subset \mathcal{N}_* (X)$, where the last inclusion holds by Lemma \ref{lem:projpushfwd}. Thus, $\iota_X \mathcal{A} (X) \subset \mathcal{N} (X)$, as desired. This completes the proof.
\end{proof}

Since $\Omega_* (X) \to \Omega_* ^{\num} (X)$ is surjective, we immediately deduce:
\begin{corollary}For each $X \in \Sch_k$, there is a natural surjective homomorphism $\Omega_* ^{\alg} (X) \to \Omega_* ^{\num} (X)$. 
\end{corollary}

Combining the above discussions, we conclude:

\begin{theorem}

The natural surjections of functors $\Omega_* \to \Omega_* ^{\alg} \to \Omega_* ^{\num}$ are morphisms of oriented Borel-Moore homology theories on $\Sch_k$. The natural surjections of functors $\Omega^* \to \Omega^* _{\alg} \to \Omega^* _{\num}$ are morphisms of oriented cohomology theories on $\Sm_k$.

\end{theorem}

\subsection{Computation for cellular varieties} Recall that a cellular scheme $X$ is a quasi-projective $k$-scheme with a decreasing filtration $X_i$ such that the complement $X_i \setminus X_{i+1}$ is isomorphic to a disjoint union of affine spaces, called cells. 

\begin{theorem}Let $X$ be a cellular scheme over $k$. Then, $\Omega^{\num}_* (X)$ is a free $\mathbb{L}$-module on the set of cells of $X$. Furthermore, we have isomorphisms of $\mathbb{L}$-modules via the natural maps $\Omega_* (X) \overset{\sim}{\to} \Omega^{\alg} _* (X) \overset{\sim}{\to} \Omega^{\num}_* (X).$
\end{theorem}

\begin{proof}The first one is an immediate consequence of the localization (Theorem \ref{thm:localization}), the homotopy invariance (Theorem \ref{thm:EH}), and Remark \ref{rmk:point}. The second one follows by combining the first one with \cite[Theorem 1.2 (3)]{KP}.
\end{proof}

Thus, for cellular varieties, one may use the theory $\Omega_* ^{\num}$ interchangeably with $\Omega_*$, and it can offer some flexibilities. Such examples may be examined in follow-up works.

\subsection{Some extensions}\label{sec:sext}

We extend some results proven in \S \ref{sec:num} to the case when $X$ is in $\Sch_k$.

\subsubsection{Comparison with algebraic cycles}Theorem \ref{thm:num1_qp} below generalizes Theorem \ref{thm:num}(1), which assumed that $X$ is smooth projective. When $X \in \Sch_k$, how one defines numerical equivalence on algebraic cycles on $X$ is not widely discussed. In fact, as we did in Definitions \ref{defn:smqp} and \ref{defn:qp}, one can define it in steps from that of the case of smooth projective varieties using desingularizations, first for smooth quasi-projective varieties, then for a general quasi-projective schemes. However, this is in essence equivalent to \emph{define} ${\rm Num}_* (X):= \mathcal{N}_* (X) \otimes_{\mathbb{L}} \mathbb{Z}$ given that we have $\Omega_* (X) \otimes_{\mathbb{L}} \mathbb{Z} \simeq \CH_* (X)$ for each $X \in \Sch_k$ by \cite[Theorem 4.5.1]{LM}. Using it, we define $\CH_* ^{\num} (X):= \CH_* (X) / {\rm Num}_* (X)$. Thus, we immediately deduce:

\begin{theorem}\label{thm:num1_qp}Let $X \in \Sch_k$. Then, we have a natural isomorphism $\ov{\phi} ^{\num} : \Omega_* ^{\num} (X) \otimes_{\mathbb{L}} \mathbb{Z} \simeq \CH_* ^{\num} (X)$.
\end{theorem}

\subsubsection{Finiteness}

We know that $\Omega_* ^{\num} (X)$ is a finitely generated $\mathbb{L}$-module for each smooth projective $X$ by Theorem \ref{thm:num}(2). We generalize it to $\Sch_k$.

\begin{lemma}\label{lem:smqp_finrk}Let $X$ be a smooth quasi-projective $k$-variety. Then, $\Omega_* ^{\num} (X)$ is a finitely generated $\mathbb{L}$-module.
\end{lemma}

\begin{proof}Choose a smooth compactification $j: X \hookrightarrow \widetilde{X}$. Now, the part of the localization sequence $ \Omega_* ^{\num} (\widetilde{X} ) \overset{j^*}{\to} \Omega_* ^{\num}(X) \to 0$ shows that $\Omega_* ^{\num} (X)$ is a finitely generated $\mathbb{L}$-module, being the quotient of a finitely generated $\mathbb{L}$-module $\Omega_* ^{\num} (\widetilde{X})$.
\end{proof}

\begin{theorem}\label{thm:gen_finrk}Let $X \in \Sch_k$. Then, $ \Omega_* ^{\num} (X)$ is a finitely generated $\mathbb{L}$-module.
\end{theorem}

\begin{proof}
Since $\Omega_* (X) = \Omega_* (X_{\rm red})$, we may assume $X$ is reduced. We prove it by induction on $\dim X$. If $\dim X = 0$, then $X$ is a finite disjoint union of point, so the theorem holds. Suppose $\dim X >0$. Suppose that the result holds for schemes of dimension $< \dim X$. Let $X_{\rm sm} \subset X$ be the smooth locus, and let $Y := X \setminus X_{\rm sm}$. Then, the localization sequence of Theorem \ref{thm:localization} gives an exact sequence $\Omega_* ^{\num} (Y) \to \Omega_* ^{\num} (X) \to \Omega_*^{\num} (X_{\rm sm}) \to 0$ of $\mathbb{L}$-modules, where $\Omega_* ^{\num} (X_{\rm sm})$ is finitely generated as an $\mathbb{L}$-module by Lemma \ref{lem:smqp_finrk}, while $\Omega_*^{\num} (Y)$ is finitely generated as an $\mathbb{L}$-module by the induction hypothesis. So, we deduce from the exact sequence that so is $\Omega_* ^{\num}(X)$. This completes the proof.
\end{proof}

\section{Homological equivalence}
\label{sec:hom}

In this section, we show that homologically trivial cobordism cycles are numerically trivial. Here, homological equivalence is considered for the complex cobordism ${\rm MU}$ and the \'etale cobordism $\widehat{\rm MU}_{\et}$. We then study the integral cobordism analogue of the standard conjecture $(D)$.

\subsection{Complex cobordism}

Let $X$ be a smooth projective variety over $k$. In \cite{Quillen}, Quillen defined a notion of complex oriented cohomology theories on the category of differentiable manifolds and showed that the complex cobordism theory $X \to {{\rm MU}}^*(X)$ can be interpreted as a universal complex oriented cohomology theory. For any embedding $\sigma:k\hookrightarrow\C$, we have  a canonical morphism of graded rings $ \Phi^{{\rm top}}:\Omega^*(X)\longrightarrow {{\rm MU}}^{2*}(X_\sigma(\C))$ by the universality of algebraic cobordism. 
In \cite[\S 9]{KP}, $\ker(\Phi^{{\rm top}})$ is defined to be the group of cobordism cycles \textit{homologically equivalent to }0. 
\begin{theorem}
Let $X$ be a smooth projective variety over a field $k$ with an embedding $\sigma:k\hookrightarrow\C$. Under the above definition, a homologically trivial cobordism cycle is numerically trivial.
\end{theorem}

\begin{proof}
Let $\pi: X \to \Spec (k)$ be the structure map. The map $\Phi^{{\rm top}}$ is a ring homomorphism and commutes with push-forwards. Thus, we have the following commutative square:
\[\xymatrix{\Omega^*(X) \ar[r]^{\pi_*} \ar[d]_{\Phi^{{\rm top}}} & \Omega^{*+d}(k) \ar[d]^{\Phi^{{\rm top}}}_\wr \\
{{\rm MU}}^{2*}(X_\sigma(\C)) \ar[r]^{\pi_*} & {{\rm MU}}^{2*+2d}(pt),} \]
where $d:=\dim_k X$. Suppose $\alpha\in\Omega^*(X)$ is homologically trivial, \emph{i.e.\ }$\Phi^{{\rm top}}(\alpha)=0$. For any $\gamma\in\Omega^*(X)$, we have $\Phi ^{\rm top} (\pi_*(\alpha\cdot\gamma))=^{\dagger} \pi_*(\Phi^{{\rm top}}(\alpha\cdot\gamma))= ^{\ddagger} \pi_* (\Phi^{\rm top} (\alpha) \cdot \Phi^{\rm top} (\gamma)) = \pi_* (0 \cdot \Phi ^{\rm top} (\gamma)) = 0$, where $\dagger$ holds by the above commutative diagram, and $\ddagger$ holds because $\Phi^{\rm top}$ is a ring homomorphism. Since $\Phi^{{\rm top}}:\Omega^*(k)\to {{\rm MU}}^{2*}(pt)$ is an isomorphism by \cite[Corollary 1.2.11(1)]{LM}, we deduce that $\pi_* (\alpha \cdot \gamma) = 0$. This means, $\alpha$ is numerically trivial.
\end{proof}

\subsection{Etale cobordism}

Let $\ell$ be a prime. Quick defined a notion of \'etale cobordism in \cite[Definition~4.11]{Quick} on $\Sm_k$ and showed that the \'etale cobordism ${{{\rm MU}}}^{2*}_{\et}(-;\Z/\ell)$ is an oriented cohomology theory in \cite[Theorem~5.7]{Quick}. Hence, there is a canonical morphism of oriented cohomology theories
$ \theta:\Omega^*(X)\to\Omega^*(X)\otimes_\Z\Z/\ell \to {{{\rm MU}}}^{2*}_{\et}(X;\Z/\ell). $
He also showed that for a separably closed field $k$  of characteristic $0$, the morphism
$ \Omega^*(k;\Z/\ell ^\nu)\to {{\rm MU}}^{2*}_{\et}(k;\Z/\ell^\nu)$ is an isomorphism. 
This implies that we have an isomorphism
$ \Omega^*(k;\Z_\ell)\overset{\theta_\ell}{\to}\widehat{{{\rm MU}}}^{2*}_{\et}(k;\Z_\ell)$,
where $\Z_\ell$ is the ring of $\ell$-adic integers. Thus, we have the following commutative diagram:
\begin{equation}\label{eq:etale}
\xymatrix{ \Omega^*(X) \ar[r]^{\pi_*} \ar[d] \ar@/_3pc/[dd]_\theta & \Omega^{*+d}(k) \ar[d]^{\eta_\ell} \ar@/^3.3pc/[dd]^\theta \\ \Omega^*(X)\otimes_\Z\Z_\ell \ar[r]^{\pi_*} \ar[d] & \Omega^{*+d}(k)\otimes_\Z\Z_\ell \ar[d]^-{{\theta_\ell}}_-\wr \\
\widehat{{{\rm MU}}}^{2*}_{\et}(X;\Z_\ell) \ar[r]^{\pi_*} & \widehat{{{\rm MU}}}^{2*+2d}_{\et}(k;\Z_\ell)}
\end{equation}

Note that the natural map $\eta_\ell$ is injective. We may use \'etale cobordism $\widehat{{{\rm MU}}}_{\et}$ to define another homological equivalence of cobordism cycles. A cobordism cycle is said to be homologically trivial with respect to $\widehat{{{\rm MU}}}_{\et}$ if it is in $\ker(\theta)$. As in the case of complex cobordism, we prove
\begin{theorem}
Let $X$ be a smooth projective variety over a field $k$ of characteristic zero, and consider the homological equivalence defined by the \'etale cobordism $\widehat{{{\rm MU}}}_{\et}$. Then, a homologically trivial cobordism cycle is numerically trivial.
\end{theorem}

\begin{proof}
Let $\alpha\in\Omega^*(X)$ be homologically trivial, \emph{i.e.\ }$\theta(\alpha)=0$. For any $\gamma\in\Omega^*(X)$, $\theta(\alpha\cdot\gamma)=\theta(\alpha)\cdot\theta(\gamma)=0. $ By the commutativity of \eqref{eq:etale}, we have $ \theta\circ\pi_*(\alpha\cdot\gamma)=\pi_*\circ\theta(\alpha\cdot\gamma)=0.$
But $\theta=\overline{\theta}\circ\eta_\ell$ and $\overline{\theta}$ is an isomorphism. Thus, $\eta_\ell(\pi_*(\alpha\cdot\gamma))=0.$
Since $\eta_\ell$ is injective, $\pi_*(\alpha\cdot\gamma)=0$, which means $\alpha$ is numerically trivial.
\end{proof}

\subsection{On the standard conjecture $(D)$}
Recall \emph{the standard conjecture $(D)$} (\emph{cf.} \cite{Kleiman}) asserts that when $X$ is a smooth projective variety, homological equivalence and numerical equivalence for algebraic cycles with $\mathbb{Q}$-coefficients on $X$ coincide. Certain partial results are known. For instance, when the base field is $\mathbb{C}$, it is known when $X$ is an abelian variety or $\dim X \leq 4$ by \cite{Lieberman}. When $X$ is an abelian variety over a finite field, it is proven by \cite{Clozel}. Because smash-nilpotence on algebraic cycles with $\mathbb{Q}$-coefficients implies homological triviality, the results \cite{KS} and \cite{Sebastian} on equivalence of smash-nilpotence and numerical equivalence for certain cycles with $\mathbb{Q}$-coefficients on certain varieties also imply special cases of the standard conjecture $(D)$.

Since the classical homological equivalence comes from the kernel of the cycle class map $cl : \CH^* (X) \to H^*(X)$ on algebraic cycles for a choice of a Weil cohomology theory $H^*(-)$, while we now have the generalized cycle class maps $cl_{\Omega} = \Phi^{\rm top}: \Omega^* (X) \to \MU^* (X)$ and $\Omega^* (X) \to \widehat{\MU}_{\et} (X)$ on cobordism cycles that lift $cl$, it is an interesting question to formulate a cobordism analogue of the standard conjecture $(D)$, but \emph{without} taking the $\mathbb{Q}$-coefficients, for cobordism cycles often capture geometric information on torsions somewhat better. 

We briefly formulate and discuss generalizations of the standard conjecture $(D)$ for cobordism cycles with $\mathbb{Z}$-coefficients. Here, we call them still as `standard conjectures' but it shouldn't be taken as a sign that we believe these. In fact, we will derail from this expectation as we will see in Theorem \ref{thm:h>nOmega}. 

We state it when $k=\mathbb{C}$ to use the complex cobordusm $\MU$, but one may use any field $k$ of characteristic $0$ replacing $\MU$ by the \'etale cobordism, and $H^*(-, \mathbb{Z})$ by the $\ell$-adic \'etale cohomology. Observe that we have a commutative diagram
\begin{equation}\label{eqn:homonum}
\xymatrix{
\Omega^* (X) \ar[d] ^{cl_{\Omega}} \ar[rr] ^{\phi} & & \CH^* (X) \ar[dl] _{cl_T} \ar[d] ^{cl} \\
\MU^* (X) \ar[r] ^{\psi \ \ \ \ } & \MU ^* (X) \otimes_{\mathbb{L}} \mathbb{Z} \ar[r] ^{\ \ \ \ \gamma} & H^* (X, \mathbb{Z})}\end{equation}
of three versions of cycle class maps, where $cl_T$ is the cycle class map of Totaro in \cite[Theorems 3.1, 4.1]{Totaro}. While $\phi$ induces the isomorphism $\Omega^* (X) \otimes_{\mathbb{L}} \mathbb{Z} \simeq \CH^* (X)$ by \cite[Theorem 1.2.19]{LM}, and $\psi$ induces the identity map of $\MU^* (X) \otimes_{\mathbb{L}} \mathbb{Z}$, the map $\gamma$ is neither injective nor surjective in general, although it becomes an isomorphism when $H^* (X, \mathbb{Z})$ is torsion free.

\begin{definition}Let $X$ be a smooth projective variety over $\mathbb{C}$.
\begin{enumerate}
\item For the cycle class map $cl_{\Omega}: \Omega^* (X) \to \MU^* (X)$, let $\mathcal{H}^*_{\MU} (X) = \ker (cl_{\Omega})$. The statement $(D)_{\MU}$ is the assertion that $\mathcal{H}^*_{\MU} (X) = \mathcal{N}^*(X)$.

\item For the cycle class map $cl_{T}: \CH^* (X) \to \MU^* (X) \otimes_{\mathbb{L}} \mathbb{Z}$, let $\mathcal{H}^* _{\rm T} (X) = \ker (cl_T)$. The statement $(D)_{\rm T}$ is the assertion that $\mathcal{H}_{\rm T} ^* (X) = {\rm Num}^* (X)$.

\item For the usual cycle class map $cl: \CH^* (X) \to H^* (X, \mathbb{Z})$, let $\mathcal{H}^* _{\rm H} (X) = \ker (cl)$. The statement $(D)_{\rm H}$ is the assertion that $\mathcal{H}_{\rm H} ^* (X)= {\rm Num} ^* (X)$. 
\end{enumerate}
\end{definition}
The classical standard conjecture $(D)$ is exactly the statement $(D)_{\rm H} \otimes \mathbb{Q}$, where $\otimes \mathbb{Q}$ means we take the tensor product with $\mathbb{Q}$ on both sides of the cycle class map. Since all cycle class maps $cl_{\Omega}$, $cl_T$ and $cl$ are ring homomorphisms, the kernels $\mathcal{H}_{\MU} ^* (X) \subset \Omega^* (X)$, and $\mathcal{H}_{\rm T} ^* (X)\subset \mathcal{H}_{\rm H} ^* (X) \subset \CH^* (X)$ are ideals, where the inclusion $\mathcal{H}_{\rm T} ^* (X) \subset \mathcal{H}_{\rm H} ^* (X)$ follows from the right commutative triangle of \eqref{eqn:homonum}. We begin with the following simple result, similar to Corollary \ref{cor:stpdstc_hn}:

\begin{lemma}\label{lem:stpdab_mut}The map $\phi: \Omega ^* (X) \to \CH^* (X)$ induces $\mathcal{H}_{\MU} ^* (X) \otimes_{\mathbb{L}} \mathbb{Z}  \overset{\sim}{\to} \mathcal{H}_{\rm T} ^* (X)$.
\end{lemma}

\begin{proof}Let $\alpha \in \mathcal{H}_{\MU} ^* (X)$, i.e. $cl_{\Omega} (\alpha) = 0$. So, $0 = \phi (0) = \phi \circ cl_{\Omega} (\alpha) =^{\dagger} cl_T \circ \phi (\alpha)$, where $\dagger$ holds by the commutative diagram \eqref{eqn:homonum}, so that $\phi (\alpha) \in \mathcal{H}_{\rm T} ^* (X)$. Hence, $\phi (\mathcal{H}_{\MU} ^* (X)) \subset \mathcal{H}_{\rm T} ^* (X)$. But, by \cite[Theorem 1.2.19]{LM}, $\phi$ is the quotient map ${\rm Id} \otimes_{\mathbb{L}} \mathbb{Z}= {\rm Id}/\mathbb{L}_{>0} \cdot {\rm Id}$, so $\mathcal{H}_{\MU} ^* (X) \otimes_{\mathbb{L}} \mathbb{Z} \subset \mathcal{H}_{\rm T} ^* (X)$. 

Conversely, let $\beta \in \mathcal{H}_{\rm T} ^* (X)$. Since $\phi$ is surjective, there is $\alpha \in \Omega^* (X)$ with $\phi (\alpha) = \beta$. Since $cl_T (\beta ) = 0$ by definition, by the commutativity of the left square in \eqref{eqn:homonum}, we have $\alpha \in \ker (\psi \circ \cl _{\Omega})$. So, $cl_{\Omega} (\alpha) \in \ker (\psi) = \mathbb{L}_{>0} \cdot \MU^* (X)$. Since $cl_{\Omega}$ is a graded $\mathbb{L}$-module homomorphism as well as a homomorphism of rings, we have $cl_{\Omega} ^{-1}  ( \mathbb{L}_{>0} \cdot \MU^* (X)) = \mathcal{H}_{\MU} ^* (X) + \mathbb{L}_{>0} \cdot \Omega^* (X)$, i.e. $\alpha \in \mathcal{H}_{\MU}^* (X) + \mathbb{L}_{>0} \cdot \Omega ^* (X)$. Write $\alpha = \alpha ' + \alpha ''$ with $\alpha' \in \mathcal{H}_{\MU} ^* (X)$ and $\alpha'' \in \mathbb{L}_{>0} \cdot \Omega^* (X)$. Since $\phi$ kills all of $\mathbb{L}_{>0} \cdot \Omega^* (X)$, we have $\beta = \phi (\alpha) = \phi (\alpha')$. This means, $\beta \in \phi (\mathcal{H}_{\MU}^* (X))$. But, $\phi$ is the quotient map ${\rm Id}/ \mathbb{L}_{>0} \cdot {\rm Id}$, so, $\mathcal{H}_{\MU} ^* (X) \otimes_{\mathbb{L}} \mathbb{Z} \supset \mathcal{H}_{\rm T} ^* (X)$. 
\end{proof}

\begin{proposition}\label{prop:standard main}Let $X$ be a smooth projective variety over $\mathbb{C}$.
\begin{enumerate}
\item $(D)_{\MU}$ is equivalent to $(D)_{\rm T}$.
\item $(D)_{\rm T}$ implies $(D)_{\rm H}$.
\item If $\gamma: \MU^* (X) \otimes_{\mathbb{L}} \mathbb{Z} \to H^* (X, \mathbb{Z})$ is an isomorphism, then $(D)_{\rm T}$ is equivalent to $(D)_{\rm H}$. In particular, when $H^* (X, \mathbb{Z})$ is torsion-free, all three statements are equivalent.
\item $(D)_{\rm T} \otimes \mathbb{Q}$ is equivalent to $(D)_{\rm H} \otimes \mathbb{Q}$. In particular, all three statements are equivalent with $\mathbb{Q}$-coefficients.
\end{enumerate}
\end{proposition}

\begin{proof}(1) Suppose $(D)_{\MU}$ holds. Then, we have $\mathcal{H}_{\MU} ^* (X) = \mathcal{N}^* (X)$. Then, $\mathcal{H}_{\rm T} ^* (X) = ^{\dagger} \mathcal{H}_{\MU} ^* (X) \otimes_{\mathbb{L}} \mathbb{Z} = \mathcal{N}^* (X) \otimes_{\mathbb{L}} \mathbb{Z} =^{\ddagger} {\rm Num} ^* (X)$, where $\dagger$ holds by Lemma \ref{lem:stpdab_mut} and $\ddagger$ holds by Corollary \ref{cor:stpdstc_hn}, proving $(D)_{\rm T}$.

Conversely, suppose $(D)_{\rm T}$ holds. Consider the $\mathbb{L}$-module $M^*:= \mathcal{N}^* (X) / \mathcal{H}_{\MU}^* (X)$. Then, $M^* \otimes_{\mathbb{L}} \mathbb{Z} = ( \mathcal{N}^* \otimes_{\mathbb{L}} \mathbb{Z}) / ( \mathcal{H}_{\MU} ^* (X) \otimes_{\mathbb{L}} \mathbb{Z}) =^{\dagger} {\rm Num} ^* (X) / \mathcal{H}_{\rm T}^* (X) = ^{\ddagger}  0$, where $\dagger$ holds by Corollary \ref{cor:stpdstc_hn} and Lemma \ref{lem:stpdab_mut}, and $\ddagger$ holds by $(D)_{\rm T}$. But, by \cite[Lemma 9.8]{KP}, this implies $M^*$ is zero.

(2) By $(D)_{\rm T}$ we have $\mathcal{H}_{\rm T} ^* (X) = {\rm Num}^* (X)$. We know $\mathcal{H}_{\rm T} ^* (X) \subset \mathcal{H}_{\rm H} ^* (X)$ by the commutative triangle in \eqref{eqn:homonum}. But, $cl$ is a ring homomorphism, so $\mathcal{H}_{\rm H} ^* (X) \subset {\rm Num} ^* (X)$. Hence, $\mathcal{H}_{\rm T} ^* (X) = \mathcal{H}_{\rm H} ^* (X) = {\rm Num} ^* (X)$. 

(3) When $\gamma$ is an isomorphism, then obviously $(D)_{\rm T}$ is equivalent to $(D)_{\rm H}$. When $H^* (X, \mathbb{Z})$ is torsion-free, the map $\gamma$ is an isomorphism, so that by the first part, $(D)_{\rm T}$ and $(D)_{\rm H}$ are equivalent. Hence by (1), all three statements are equivalent.

(4) This is obvious because both the kernel and the cokernel of the homomorphism $\gamma : \MU^* (X) \otimes_{\mathbb{L}} \mathbb{Z} \to H^* (X, \mathbb{Z})$ in \eqref{eqn:homonum} are torsion, so that with $\mathbb{Q}$-coefficients $\gamma \otimes \mathbb{Q}$ becomes an isomorphism.
\end{proof}

The above proposition implies that, if we ignore the torsions everywhere, then the strongest statement $(D)_{\MU}$ we introduced here is equivalent to the classical standard conjecture $(D)$. Thus, one may hope to draw new geometric informations using the statement $(D)_{\MU}$ without ignoring the torsions on the (cobordism / algebraic) cycle groups and the cohomology groups. For such torsions, here is one important result:

\begin{theorem}[{Totaro \cite[Theorem 7.2]{Totaro}}]\label{thm:torsiontotaro}There exists a smooth projective variety $X$ over $\mathbb{C}$ of dimension $15$ and a cycle $\alpha \in  \CH^3 (X)$, not algebraically trivial, such that $2 \alpha = 0 \in \CH^3 (X)$, $cl_T (\alpha)\not = 0$ in $\MU^6 (X) \otimes_{\mathbb{L}} \mathbb{Z}$, while $cl (\alpha) = 0$ in $H^6 (X, \mathbb{Z})$. 
\end{theorem}

We can deduce that the statement $(D)_{\MU}$ may not hold in general:

\begin{theorem}\label{thm:h>nOmega}There is a smooth projective variety $X$ over $\mathbb{C}$ for which the standard conjecture $(D)_{\MU}$ does not hold. In particular, the homological equivalence is finer than the numerical equivalence for algebraic cobordism $\Omega^*$ in general.
\end{theorem}

\begin{proof}If $(D)_{\MU}$ holds for $X$, then by Proposition \ref{prop:standard main}(1)(2), we have $(D)_{\rm H}$ for $X$. In particular, we have the equality $\mathcal{H}_{\rm T} ^* (X) = \mathcal{H}_{\rm H} ^* (X) = {\rm Num} ^* (X)$ in $\CH^* (X)$. However, for the variety $X$ and the cycle $\alpha \in \CH^3 (X)$ in Theorem \ref{thm:torsiontotaro}, we have $\alpha \in \mathcal{H}_{\rm H} ^* (X)$, while $\alpha \not \in \mathcal{H}_{\rm T} ^* (X)$, contradiction. Thus, the statement $(D)_{\MU}$ is false for this $X$. 
\end{proof}

\begin{remark}\label{rmk:h>nOmega}If we work with smooth projective varieties over a field $k$ of characteristic $0$, then we can repeat the arguments of this section replacing the complex cobordism $\MU$ by the \'etale cobordism $\widehat{\MU}_{{\et}}$ and the singular cohomology by the $\ell$-adic \'etale cohomology. In this case, one can replace Theorem \ref{thm:torsiontotaro} by \cite[Proposition 5.3(2)]{Quick2}, which proves the same result with $\mathbb{C}$ replaced by $k$. Hence, Theorem \ref{thm:h>nOmega} holds for $k$ as well, where the homological equivalence in this case is given by the \'etale cobordism. So, we can conclude that for algebraic cobordism, homological equivalence and numerical equivalence are distinguished in general, and the situation is more complicated than that of algebraic cycles with $\mathbb{Q}$-coefficients.
\end{remark}

Here is one question, that is natural to ask at this point:

\begin{question}For a smooth projective variety $X$ over $\mathbb{C}$, does the homological equivalence given by $\MU^*$ coincide with the one given by $\widehat{\MU}_{\et} ^*$?
\end{question}

We do not know how to answer this one yet. Since both of the two homological equivalences are now strictly finer than the numerical equivalence on algebraic cobordism in general, one deduces that the integral version of the Voevodsky conjecture (whether the integral version of the smash-nilpotence is equivalent to the numerical equivalence) on algebraic cobordism is also false in general. If the integral version of the Voevodsky conjecture was indeed true, then we could have deduced that both homological equivalences are equivalent by ``sandwich principle'', but we can no longer argue like this. This question suggests that for algebraic cobordism somewhat finer approach than that of the idea of the Voevodsky conjecture on algebraic cycles, should be devised in the future. We have nothing to say about it at this moment, except that, we show with $\mathbb{Q}$-coefficients, the Voevodsky conjecture on algebraic cobordism is somewhat hopeful. The rest of the paper is devoted to studying this question in some special cases.

\section{Voevodsky conjecture on algebraic cobordism}\label{sec:new Voevodsky}

\subsection{Smash nilpotence}\label{sec:smash}

We recall the definition of $\mathbb{Q}$-smash nilpotence:
\begin{definition}[{\cite[Definition 10.1]{KP}}]
Let $X\in\Sch_k$. A cobordism cycle $\alpha\in \Omega_*(X)$ is said to be \emph{rationally smash nilpotent} if, for some positive integer $N$, $\alpha^{\boxtimes N}:=\alpha\times\cdots\times\alpha=0$ in $\Omega_*(X^N)_\Q$.
\end{definition}

Since we will be working with algebraic cobordism with $\Q$-coefficients from now on, we would say ``smash nilpotent'' to mean ``rationally smash nilpotent''.

\begin{theorem}
Let $X$ be a smooth projective variety over a field $k$ with an embedding $\sigma:k\hookrightarrow\C$. Then, smash nilpotent cycles in $\Omega^*(X)_\Q$ are homologically trivial.
\end{theorem}

\begin{proof}
In \cite[p.471]{Totaro}, it is shown that ${{\rm MU}}^*(X)_\Q$ is a free $\bL_*\otimes\Q$-module generated by any set of elements that map to a basis of $H^*(X;\Q)$. This, along with \cite[Theorem~44.1]{CF} shows that for smooth projective varieties $X,\, Y$, the homomorphism
$\chi : {{\rm MU}}^*(X)_{\mathbb{Q}}\otimes_{\bL_* \otimes \Q} {{\rm MU}}^*(Y)_\Q\to {{\rm MU}}^*(X\times Y)_\Q$
is an isomorphism. Note also that, by definition, if $\alpha\in \Omega^*(X)$,  then $\theta(\alpha^{\boxtimes N})\in {{\rm MU}}^*(X^N)$ equals $\chi(\theta(\alpha)^{\otimes N})$. Thus, if $\alpha$ satisfies $\alpha^{\boxtimes N}=0\in\Omega^*(X^N)_\Q$ for some $N$, then $\theta(\alpha)^{\otimes N}=0\in {{\rm MU}}^*(X)^{\otimes N}_\Q$. From \cite[Section 6]{Quillen}, we know that $\bL_*\otimes\Q$ is isomorphic to a polynomial ring over $\Q$, and hence it has no non-zero nilpotent elements. Since ${{\rm MU}}^*(X)_\Q$ is a free $\bL_*\otimes\Q$-module, $\theta(\alpha)^{\otimes N}=0$ implies that $\theta(\alpha)=0.$\end{proof}

\begin{theorem}\label{thm:snnt}
Let $X$ be a smooth projective variety over a field $k$ of characteristic $0$. Then, smash nilpotent cycles in $\Omega^*(X)_\Q$ are numerically trivial.
\end{theorem}

\begin{proof}
Note that if $\beta\in \Omega^*(k)_\Q$ is smash-nilpotent, then it is nilpotent. However, $\Omega^*(k)_\Q \overset{\sim}{\to}\bL_* \otimes \Q$, which is a polynomial ring over $\Q$, and thus has no nonzero nilpotent. Thus, $\beta=0$. Let $\pi:X\to\Spec k$ be the structure map. If $\alpha\in\Omega^*(X)_\Q$ is smash nilpotent, then for each $\gamma\in\Omega^*(X)_\Q$, the product $\alpha\cdot\gamma$ is smash-nilpotent by \cite[Lemma~10.2]{KP} since the push-forward and pull-back maps respect external products. This implies that $\pi_*(\alpha\cdot\gamma)$ is smash-nilpotent in $\Omega^*(k)_\Q$, hence $\pi_*(\alpha\cdot\gamma)=0$, which means $\alpha$ is numerically trivial.
\end{proof}

We study the converse of Theorem \ref{thm:snnt} in Section \ref{sec:Sebastian}. Specifically, we look at the cases, where the results of Kahn-Sebastian \cite{KS} and Sebastian \cite{Sebastian} are proven for algebraic cycles. We develop some cobordism analogues of known results for algebraic cycles.

\subsection{Kimura finiteness on cobordism motives}\label{sec:Kimura}

From now on, we will work on the category $\SmProj_k$ of smooth projective varieties over an algebraically closed field $k$ of characteristic $0$, and consider algebraic cobordism with $\Q$-coefficients. 

\subsubsection{Cobordism motives}
In \cite[\S 5-6]{NZ}, for an oriented cohomology theory $A^*$, Nenashev and Zainoulline constructed the \emph{$A$-motive} of a smooth projective variety $X$, following the ideas of \cite{Manin}. We briefly recall its construction. Given $A^*$, we have the category of $A$-correspondences, denoted as $\Cor_A$, where
\begin{itemize}
\item $Ob(\Cor_A):=Ob(\SmProj_k)$,
\item $\Hom_{\Cor_A}(X,Y):=A^*(X\times Y)$ and
\item for $\alpha\in A^*(X,Y)$ and $\beta\in A^*(Y,Z)$, we have $\beta\circ\alpha := (p_{XZ})_*(p_{XY}^*(\alpha)\cdot p_{YZ}^*(\beta))$ in $ A^*(X\times Z).$
\end{itemize}
We have a functor $c:\SmProj_k^{op}\to \Cor_A$ given by $c(X)=X$ and $c(f)=(\Gamma_f)_*(1_{A(X)})\in A^*(Y\times X)$ for a morphism $f:X\to Y$, where $\Gamma_f = (f, {\rm Id}):X\to Y\times X$ is the graph of $f$. The grading on $A^*$ induces a grading on $\Hom_{\Cor_A}$ given by $\Hom_{\Cor_A}^n(X,Y):=\oplus_iA^{n+d_i}(X_i\times Y),$ where the $X_i$'s are the irreducible components of $X$ and $d_i=\dim X_i$, making $\Hom_{\Cor_A}$ into a graded algebra under composition. 

\begin{definition}\label{def:motive}
Consider the category $\Cor_A^0$ with $\Hom_{\Cor_A^0}(X,Y):= \Hom_{\Cor_A}^0(X,Y)$. The pseudo-abelian completion of $\Cor_A^0$ is called the \emph{category of effective $A$-motives}, denoted by $\mc{M}_A^\eff$. Thus, the objects in $\mc{M}_A^\eff$ are pairs $(X,p)$ where $X\in Ob(\Cor_A)$ and $p\in \Hom_{\Cor_A^0}(X,X)$ is a projector, and
\[ \Hom_{\sM_A^\eff}((X,p),(Y,q)) = \dfrac{\{\alpha\in\Hom_{\Cor_A^0}(X,Y)\vert\alpha\circ p=q\circ\alpha\}}{\{\alpha\in\Hom_{\Cor_A^0}(X,Y)\vert\alpha\circ p=q\circ\alpha=0\}}. \]
The \emph{category of $A$-motives}, denoted by $\mc{M}_A$, has the triplets $(X,p,m)$ as objects, where $(X,p)$ is an object in $\mc{M}_A^\eff$ and $m\in\Z$. The morphisms are defined as:
\[ \Hom_{\sM_A}((X,p,m),(Y,q,n)) = \dfrac{\{\alpha\in\Hom_{\Cor_A}^{n-m}(X,Y)\vert\alpha\circ p=q\circ\alpha\}}{\{\alpha\in\Hom_{\Cor_A}^{n-m}(X,Y)\vert\alpha\circ p=q\circ\alpha=0\}}. \]
\end{definition}
Note that, this means ${\rm Id}_{(X,p,0)}={\rm Id}_X=p\in\Hom_{\sM_A}((X,p,0),(X,p,0))$. The motive $(X,{\rm Id}_X,0)$ is called the \emph{motive of $X$} and denoted by $h_A(X)$.

\subsubsection{Finite dimensionality of cobordism motives}
Following \cite[Section~3]{Kimura}, we define finite-dimensionality of $\Omega^*_\Q$-motives. Each partition $\lambda=(\lambda_1,\ldots,\lambda_k)$ of an integer $n\geq 1$, determines an irreducible representation $W_\lambda$ of $\Sigma_n$ over $\Q$. The group $S_n^{op}$ acts on the $n$-fold product $X^n$ of a smooth projective variety $X$ by $\sigma(x_1,\ldots,x_n):=(x_{\sigma(1)},\ldots,x_{\sigma(n)})$ for $\sigma\in \Sigma_n $. Let $f_\sigma:X^n\to X^n$ be the morphism associated to the action of $\sigma$ and define $d_\lambda\in \Cor_{\Omega^*_\Q}(X^n,X^n)$ to be $ d_\lambda:=(\dim  \ W_\lambda/{n!}) \cdot \sum_{\sigma\in \Sigma_n}\chi_{W_\lambda}(\sigma)\cdot c(f_\sigma)^t.$
Using the properties of $W_\lambda$ and the fact that $c$ is a functor, we get $\sum d_\lambda=c({\rm Id}_{X^n})$, $d_\lambda\circ d_\lambda = d_\lambda$ and $d_\lambda\circ d_\mu=0$ whenever $\lambda\neq\mu$. Thus, we have $h(X^n)\simeq\oplus_\lambda \mathbb{T}_\lambda X^n$ in $\mc{M}_{\Omega^*_\Q}$, where $\mathbb{T}_\lambda X^n:= (X^n,d_\lambda,0)$.

\begin{definition}
For a $\Omega^*_\Q$-motive $M=(X,p,m)$, we define $M^{\otimes n}:=(X^n, p^{\boxtimes n},mn)$. If $f:M\to N$ is a morphism of motives for $M=(X,p,m_1)$ and $N=(Y,q,m_2)$, we have the morphism 
$ f^{\otimes n}:M^{\otimes n}\to N^{\otimes n}$ defined to be $f^{\otimes n}:=f^{\boxtimes n}\in \Omega^*(X^n\times Y^n)_\Q.$
\end{definition}

Let $M=(X,p,m)$ be a motive. It follows by direct computation that $c(f_\sigma)^t$ and $p^{\boxtimes n}$ commute with each other. This implies that $d_\lambda\circ p^{\boxtimes n}=p^{\boxtimes n}\circ d_\lambda$ is idempotent. Thus, we have the following definition:

\begin{definition}
For a motive $M=(X,p,m)$, we define $\mathbb{T}_\lambda M$ to be the motive $(X^n,d_\lambda\circ p^{\boxtimes n},mn)$. When $\lambda=(n)$, we denote 
$\mathbb{T}_{(n)} M$ by $\Sym^nM$ and for $\lambda=(1,1,\ldots,1)$, we denote $\mathbb{T}_{(1,1,\ldots,1)} M$ by $\wedge^nM$.

A motive $M$ is {\em evenly (resp. oddly) finite dimensional} if there exists a positive integer $N$, such that $\wedge^NM=0$ (resp. $\Sym^NM=0$). $M$ is said to be {\em finite-dimensional} if it can be written as a direct sum $M=M_+\oplus M_-$ such that $M_+$ is evenly finite dimensional and $M_-$ is oddly finite dimensional.
\end{definition}

\begin{remark}\label{rmk:Lef fd}
Let $\boldsymbol{L} := (pt, {\rm Id}_{pt}, -1)$. This is called the \emph{Lefschetz motive}. Note that, for any $i\geq 0$, $\wedge^2\boldsymbol{L}^i=0$ since $\Sigma_2 ^{\rm op}$ acts trivially on $\boldsymbol{L}^i$. Thus, $\boldsymbol{L}^i$ is evenly $1$-dimensional.
\end{remark}

The following  result of \cite{Kimura} holds in the case of $\Omega^*_\Q$-motives by the same arguments as in the proof of \cite[Proposition~6.1]{Kimura}, since the proof uses formal properties of the construction above, rather than properties specific to the Chow ring.

\begin{proposition}\label{prop:morphismparity}
Any morphism between motives of different parity is smash-nilpotent.
\end{proposition}

\begin{proposition}\label{prop:chowmot}
Let $\widetilde{\phi}:\sM_\Omega	\to\sM_{\CH}$ be the map induced by the canonical morphism $\phi:\Omega^*\to \CH^*$. If $M$ is an $\Omega^*_\Q$-motive such that $\widetilde{\phi}(M)$ is evenly (resp.\  oddly) finite-dimensional as a Chow motive, then, $M$ is evenly (resp.\  oddly) finite-dimensional.
\end{proposition}

\begin{proof}
Let $M=(X,p,m)$ be an $\Omega^*_\Q$-motive such that $\widetilde{\phi}(M)=(X,\phi(p),m)$ is oddly finite-dimensional as a Chow motive. That is, for some $N \geq 1$, $\Sym^N\widetilde{\phi}(M)=(X^N, (1/ N!) \cdot \sum_{\sigma\in \Sigma_N}c_{\CH}(f_\sigma)^t\circ\phi(p)^{\boxtimes N},mN)=0.$
But, $c_{\CH}(f_\sigma)=\phi(c_{\Omega}(f_\sigma))$ and $\phi(c_{\Omega}(f_\sigma)^t)\circ\phi(p)^{\boxtimes N}=\phi(c_{\Omega}(f_\sigma)^t\circ p^{\boxtimes N})$ since $\phi$ is a ring homomorphism and commutes with push-forwards and pullbacks. Thus, $\widetilde{\phi}(\Sym^NM)=0$. Vishik and Yagita showed in \cite[Corollary~2.8]{VY} that any isomorphism of Chow motives can be lifted to an isomorphism of cobordism motives. This implies that $\Sym^NM=0$. The proof is similar when $\widetilde{\phi}(M)$ is evenly finite-dimensional.
\end{proof}

By \cite[Corollary~8.2]{BH}, for an abelian variety $A$ of dimension $g$, we get a canonical decomposition of the cobordism motive of $A$, $ h_\Omega(A)=\bigoplus_{i=0}^{2g}h_\Omega^i(A),$
where $h_\Omega(A)=(A,{\rm Id}_A,0)$  and $h_\Omega^i(A)=(A,\pi_i,0)$.

\begin{corollary}\label{cor:hoddA}
The motive $h_\Omega^{{\rm odd}}(A):=\bigoplus_{r : odd}h_\Omega^r(A)$ of an abelian variety $A$ is oddly finite-dimensional.
\end{corollary}
\begin{proof}
In \cite{Shermenev}, Shermenev gave a decomposition of the Chow motive of an abelian variety $A$ of dimension $n$ as
\[ h_{\CH}(A)= \bigoplus_{i=0}^n\Sym^i(X,a_+)\oplus\bigoplus_{i=0}^{n-1}\Sym^i(X,a_+)\otimes\boldsymbol{L}^{n-i}\]
for some curve $X$ and a projector $a_+$ on $X$. It follows from \cite[Theorem~4.2]{Kimura} that the motive $(X,a_+)$ defined by Shermenev is oddly finite-dimensional. Since odd symmetric powers of an oddly finite-dimensional motive is oddly finite-dimensional, we get  that $h_{\CH}^{{\rm odd}}(A)$ is oddly finite-dimensional. From \cite[Corollary~8.2]{BH}, we get $\widetilde{\phi}(h_\Omega^{{\rm odd}}(A))=h_{\CH}^{{\rm odd}}(A)$. Thus, Proposition~\ref{prop:chowmot} implies that $h_\Omega^{{\rm odd}}(A)$ is oddly finite-dimensional.
\end{proof}

\subsection{Voevodsky's conjecture for cobordism cycles}\label{sec:Sebastian}
Let $k = \bar{k}$. Let $A$ be an abelian variety over $k$. Our objective is to prove Theorem \ref{thm:Voev} and discuss its consequences. Recall that $\beta\in\Omega^*(A)_\Q$ is called a \emph{skew cycle} if $\langle -1\rangle^*\beta=-\beta$, where, $\langle n\rangle$ denotes the endomorphism $\times n$ on $A$ for $n \in \mathbb{Z}$.

\begin{proposition}\label{prop:skew}
Any skew cycle on an abelian variety is smash-nilpotent.
\end{proposition}

We follow the sketch of \cite[Proposition~1]{KS}.

\begin{proof}Let $g = \dim A$. By \cite[Corollary~8.2]{BH}, we have the canonical decomposition, $h_\Omega(A)=\bigoplus_{i=0}^{2g}h_\Omega^i(A)$, 
where $h_\Omega^i(A)=(A,\pi_i,0)$ such that $c(\langle{n}\rangle)\circ\pi_i=n^i\pi_i=\pi_i\circ c(\langle{n}\rangle)$. By Definition~\ref{def:motive}, we have $\Omega^d(X)_\Q=\Hom(\boldsymbol{L}^d,h_\Omega(X))$, where $\boldsymbol{L}=(pt,{\rm Id}_{pt},-1)$ is the Lefschetz motive. It is easy to check that $\Hom(\boldsymbol{L}^d,h_\Omega^r(A))$ $\overset{\sim}{\to}$ $\Omega^d_{2d-r}(A)_\Q$, where the latter is defined to be $\{\alpha\in \Omega^d(A)_\Q \vert \langle{n}\rangle^*\alpha=n^r\alpha, \forall n\in\Z\}.$
Indeed, $\Hom(\boldsymbol{L}^d,h_\Omega^r(A))=\pi_{r*}\Omega^d(A)_\Q$ and $\langle{n}\rangle^*\pi_{r*}\alpha=c(\langle{n}\rangle)\circ\pi_r\circ\alpha$.

Let $\beta\in\Omega^d(A)_\Q$ be a skew cycle. Viewing $\beta$ as a morphism in $\Hom(\boldsymbol{L}^d,h_\Omega(A))$, we may write $ -\pi_r\circ\beta=\pi_r\circ(\langle -1\rangle^*\beta)=\pi_r\circ c(\langle -1\rangle)\circ\beta=(-1)^r\pi_r\circ\beta.$
Thus, $\pi_r\circ\beta=0$ for even $r$. This implies that $\beta$ factors through $h_\Omega^{{\rm odd}}(A)$ via a morphism $\beta'\in \Hom(\boldsymbol{L}^d,h_\Omega^{{\rm odd}}(A)).$
Since $\boldsymbol{L}^d$ is evenly finite-dimensional by Remark \ref{rmk:Lef fd}, it follows from Corollary~\ref{cor:hoddA} and Proposition~\ref{prop:morphismparity} that $\beta$ is smash nilpotent.
\end{proof}

Now, we will show that numerically trivial cobordism $1$-cycles on a product of curves is smash nilpotent. To achieve this, following the ideas in  \cite{Sebastian}, we show that the `modified diagonal' cobordism cycle $\Delta_c$ of Definition \ref{defn:mod diag} is smash-nilpotent in Corollary \ref{cor:DeltaSm}. We project $\Delta_c$ to a smaller product of curves and apply induction to get our desired result.

Let $Y$ be a smooth projective curve of genus $g$ and let $\Jac(Y)$ denote its Jacobian. Fix $N\geq 3$ and $m>\max\{N,2g+2\}$. By \cite[Lemma~11]{Sebastian}, there is a collection $\{r_1,r_2,\ldots,r_m\}$ such that the following conditions are satisfied.
\begin{itemize}
\item[(S1)]$\sum_{l=1}^m{m\choose l}l r_l=0$.
\item[(S2)]For every even integer $0 \leq i\leq g-1$, $\sum_{l=1}^m{m\choose l}l^{2+i}r_l =0.$
\item[(S3)]$\sum_{l=N}^m{m-N\choose l-N}r_l\neq 0$.
\end{itemize}
We may slightly modify (S2) to include $i=-2$;
\begin{itemize}
\item[(S$2'$)]For every even integer $-2 \leq i\leq g-1$, $\sum_{l=1}^m{m\choose l}l^{2+i}r_l=0.$
\end{itemize}

Let $S:=\{1,2,\ldots,m\}$ and let $p_i:Y^m\to Y$ denote the $i$-th projection. Choose a base point $y_0\in Y$. For every non-empty subset $T\subset S$, define the morphism $\phi_T:Y\to Y^m$ to be the unique one such that $p_i \circ \phi_T (y) = y$ if $i \in T$ and $p_i \circ \phi_T (y) = y_0$ if $i \not \in T$.
Let $\Delta_T$ denote the cobordism cycle $[\phi_T: Y\to Y^m]$. Let $f:Y^m\to\Sym^mY$ be the natural quotient morphism. This is a projective morphism so that $f_*$ exists on cobordism cycles. On the other hand, by \cite[Proposition~4.1]{Kimura} and \cite{Schwarzenberger}, $\Sym^m Y$ is smooth since $m \geq 2g-2$. Thus, $f$ is also an \lci ~morphism, which implies $f^*$ exists for cobordism cycles by \cite[Section~6.5.4]{LM}.

Consider the cobordism cycle $\Delta_l:=f_*\Delta_T$ for some $T\subset S$ with $\#T=l$. Note that $\Delta_l$ does not depend on the choice of $T$. In fact, $\Delta_l={m \choose l}^{-1}\cdot \sum_{\#T=l}[f \circ \phi_T: Y \to \Sym^mY]$. 

\begin{lemma}\label{lem:Delta}
$f^*\Delta_l = l!(m-l)!\sum_{\#T=l}\Delta_T$ in $\Omega_1(Y^m)_\Q$.
\end{lemma}

\begin{proof}
Clearly, $f^*\Delta_l=c\sum_{\#T=l}\Delta_T$ for some $c\in\Z$. Applying $f_*$, we get $f_*f^*\Delta_l=c\sum_{\#T=l}f_*\Delta_T=c{m\choose l}\Delta_l$. 
We will show that $\deg(f_*f^*\Delta_l) =m! \deg( \Delta_l)$, where $\deg$ is the degree of cobordism cycles on $\Sym^mY$. This would imply $c=l!(m-l)!$ as needed.

Using the projection formula, we get that $f_*f^*\Delta_l=f_*(1_{Y^m}\cdot f^*\Delta_l)=f_*(1_{Y^m})\cdot\Delta_l$. Since $\deg:\Omega_*(\Sym^mY)\to\Omega_{*-m}(k)$ is a ring homomorphism by \cite[Definition~4.4.4]{LM}, we have $\deg(f_*f^*\Delta_l) = \deg(f_*(1_{Y^m}))\deg(\Delta_l)$. 
We also know from \cite[\hbox{\it loc. cit}]{LM} that $\deg$ factorizes as $\Omega_*(\Sym^mY)\overset{i^*}{\to}\Omega_{*-m}(k(\Sym^mY)/k)\xrightarrow{(p^*)^{-1}}\Omega_{*-m}(k)$ where $i:\eta\to\Sym^mY$ denotes the generic point of $\Sym^mY$, $i^*$ is the canonical homomorphism and $p^*$ is an isomorphism by \cite[Corollary~4.4.3]{LM}. 
Since $f$ is finite and surjective and $[k(Y^m):k(\Sym^mY)]=m!$, by \cite[Lemma~4.4.5]{LM}, we have 
$i^*(f_*(1_{Y^m}))=m!1_{\Sym^mY}\in\Omega_0(k(\Sym^mY)/k)$. This implies that $\deg(f_*(1_{Y^m}))=m!(p^*)^{-1}(1_{\Sym^mY})=m!1_k\in\Omega_0(k)$. Thus, $\deg(f_*f^*\Delta_l) =m!\deg(\Delta_l)$, which completes the proof. 
\end{proof}

\begin{definition}Define $\Gamma:=\sum_{l=1}^m{m\choose l}r_l\Delta_l$ in $\Omega_1(\Sym^mY)_\Q$.
\end{definition}

\begin{lemma}\label{lem:Gamma}
$\Gamma$ is smash-nilpotent.
\end{lemma}

\begin{proof}
Since $m>2g-2$, the natural map $\Sym^mY\overset{\pi}{\to}\Jac(Y)$ is the projective bundle associated to a locally free sheaf over $\Jac(Y)$ (see \cite[Proposition~4.1]{Kimura} and \cite{Schwarzenberger}). Thus, by the definition of oriented cohomology theory (see \cite[Definition~1.1.2.$(PB)$]{LM}, we have $\Gamma=c_1(\sO(1))^{m-g-1}\cdot \pi^*\beta_0+ c_1(\sO(1))^{m-g}\cdot \pi^*\beta_1,$ for some $\beta_i\in\Omega_i(\Jac(Y))_\Q$. We first check that $\pi_*\Gamma$ is smash-nilpotent. Let $\psi:Y\to \Jac(Y)$ be the embedding using the base-point $y_0$. By \cite[Theorem~6.2]{BH}, we have a Beauville decomposition of $\Omega^{g-1}(\Jac(Y))_\Q$, giving $\psi_*(1_Y)=\sum_{-2\leqslant i\leqslant g-1}x_i$, such that $\langle{n}\rangle_*x_i=n^{2+i}x_i$, where $\langle{n}\rangle$ is the morphism $\times n$ on $\Jac(Y)$. Thus, 
\begin{equation}
\pi_*\Gamma =\sum_{l=1}^m{m\choose l}r_l\pi_*\Delta_l=\sum_{l=1}^m{m\choose l}r_l\langle{l}\rangle_*\psi_*(1_Y)=\sum_{i=-2}^{g-1}\Big(\sum_{l=1}^m{m\choose l}r_l l^{2+i}\Big)x_i.
\end{equation}
Since the $r_l$'s satisfy (S$2'$), we have that $\sum_{l=1}^m{m\choose l}r_l l^{2+i}=0$ for $i$ even. Thus, Proposition~\ref{prop:skew} implies that $\pi_*\Gamma$ is smash-nilpotent. Now, by the projection formula, $\pi_*\Gamma=\pi_*(c_1(\sO(1))^{m-g-1})\cdot \beta_0+ \pi_*(c_1(\sO(1))^{m-g})\cdot \beta_1.$
Let $\alpha_i=\pi_*(c_1(\sO(1))^{m-g+i})\in\Omega^i(\Jac(Y))_\Q$, so $\pi_*\Gamma=\alpha_{-1}\beta_0+\alpha_0\beta_1$. We also have $\pi_*\{\Gamma\cdot c_1(\sO(1))\}= \alpha_0\beta_0+\alpha_1\beta_1.$
Note that, $\phi[\pi_*\{\Gamma\cdot c_1(\sO(1))\}]=\pi_*[\phi(\Gamma)\cdot\phi \{ c_1(\sO(1)) \}],$
which is a $0$-cycle of degree $0$ by \cite[Proposition~3]{Sebastian}. Thus, by \cite[Lemma~4.5.10]{LM},  $\pi_*\{\Gamma\cdot c_1(\sO(1))\}$  is of the form $\sum n_i[\{p_i\}\to \Jac(Y)]$, where $n_i\in\Z$  and $\sum n_i=0$. But, on a smooth projective variety, such a cobordism cycle is algebraically trivial by \cite[Lemma on p.56]{Mumford} and \cite[Theorem~5.1]{KP}. Thus, $\pi_*(\Gamma\cdot c_1(\sO(1)))$ is smash-nilpotent by \cite[Theorem~10.3]{KP}.
We now use the degree formula (see \cite[Theorem~4.4.7]{LM}) to conclude that $\beta_0$ and $\beta_1$ (and hence $\Gamma$) are smash-nilpotent. The degree formula gives that
\begin{itemize}
\item $\beta_0=\sum n_i[\{p_i\}\to \Jac(Y)]$, where $n_i\in\Z$  and $p_i$'s points in $\Jac(Y)$. 
\item $\beta_1=\sum m_j[\widetilde{C_j}\to \Jac(Y)]+\sum \gamma_s[\{q_s\}\to \Jac(Y)]$, where $m_j\in\Z$, $\gamma_s\in\bL_1$, $q_s$'s are points in $\Jac(Y)$ and $\widetilde{C_j}$'s are smooth curves with projective birational morphisms $\widetilde{C_j}\to C_j\subset \Jac(Y)$.
\item $\alpha_i=\sum_{j=\max(0,-i)}^{\min(g,g-i)}\sum_{l\in K^j_i} \omega^l_{i,j}x^l_{i,j}$, where $\omega^l_{i,j}\in\bL_j$, $x^l_{i,j}\in\Omega^{i+j}(\Jac(Y))$ and $\vert K_i^j\vert <\infty$.
\end{itemize}
For $j=-i$, we have $K_i^j=\{1\}$, $x^1_{i,j}=[{\rm Id}_X]$ and it follows by \cite[Proposition~3.1(a)(i)]{Fulton} that $\omega^1_{0,0}=1$. We may write 
\begin{multline*}
\alpha_{-1}\beta_0+\alpha_0\beta_1
=(\omega^1_{-1,1}[{\rm Id}_X]+\sum_{j=2}^g\sum_{l\in K^j_{-1}}\omega^l_{-1,j}x^l_{-1,j})\sum n_i[\{p_i\}\to J] + \\
+([{\rm Id}_X]+\sum_{j=1}^g\sum_{l\in K^j_0}\omega^l_{0,j}x^l_{0,j})(\sum m_j[\widetilde{C_j}\to \Jac(Y)]+\sum \gamma_l[\{q_l\}\to \Jac(Y)])
\end{multline*}
\begin{multline*}
=\sum n_i\omega^1_{-1,1}[\{p_i\}\to J] + \sum \gamma_s[\{q_s\}\to \Jac(Y)] +\sum m_j[\widetilde{C_j}\to \Jac(Y)]+\\
\sum_j\sum_{l\in K^1_0} m_j\omega^l_{0,1}x^l_{0,1}[\widetilde{C_j}\to \Jac(Y)].
\end{multline*}
Now, by \cite[Lemma~4]{Sebastian}, $\phi(\beta_1)=\sum m_j[C_j]$ is smash-nilpotent. \cite[Lemma~4.5.3]{LM} gives us that $\sum m_j[\widetilde{C_j}\to \Jac(Y)]+\sum \mu_t[\{r_t\}\to \Jac(Y)]$ is smash-nilpotent for some points $r_l$ in $\Jac(Y)$ and $\mu_l\in\bL_1$. Multiplying with $\omega^l_{0,1}x^l_{0,1}$, we get that $\sum m_j\omega^l_{0,1}x^l_{0,1}[\widetilde{C_j}\to \Jac(Y)]$ is smash-nilpotent by \cite[Lemma~10.2(1)]{KP}. Hence, $\sum m_j\sum_{l\in K^1_0}\omega^l_{0,1}x^l_{0,1}[\widetilde{C_j}\to \Jac(Y)]$ is smash-nilpotent. Since $\alpha_{-1}\beta_0+\alpha_0\beta_1$ is smash-nilpotent, this gives $\omega^1_{-1,1}\beta_0+\beta_1$ is smash-nilpotent. Next, note that $
\alpha_0\beta_0+\alpha_1\beta_1=\beta_0+\sum_{l\in K^0_1}\omega^l_{1,0}x^l_{1,0}\sum m_j[\widetilde{C_j}\to \Jac(Y)]$ is smash nilpotent. Similarly as above, $\sum_{l\in K^0_1}\omega^l_{1,0}x^l_{1,0}\sum m_j[\widetilde{C_j}\to \Jac(Y)]$ is smash-nilpotent. This implies that $\beta_0$, and hence $\beta_1$ are smash-nilpotent, which completes the proof.
\end{proof}

\begin{definition}\label{defn:mod diag}
Define the modified diagonal cobordism cycle to be $\Delta_c:= (1/ m!) \cdot f^*\Gamma$. 
\end{definition}

By Lemma~\ref{lem:Delta}, $\Delta_c=\sum_{l=1}^m r_l (\sum_{T\subset S,\#T=l }\Delta_T ).$ Then, Lemma \ref{lem:Gamma} shows that
\begin{corollary}\label{cor:DeltaSm}
$\Delta_c$ is smash-nilpotent.
\end{corollary}

Now, let $X:=C_1\times C_2\times\cdots\times C_N$ be a product of $N$ smooth projective curves. Let $Y$ be a smooth projective curve with a morphism $j:Y\to X$. Let $q_i:X\to C_i$ denote the projection onto the $i$-th factor. Define a morphism $\psi:Y^m\to X$ as 
$ Y^m\overset{pr}{\to} Y^N\overset{j^N}{\to} X^N\xrightarrow{q_1\times\cdots\times q_N}X$,
where $pr$ is the projection to the first $N$ coordinates. Let $S_0:=\{1,2,\ldots,N\}$. For a closed point $v=(v_1,\ldots,v_N)$ of $X$ and a subset $T\subset S_0$, we define $\zeta^v_{T}:X\to X$ to be the unique morphism such that $q_i \circ \zeta^v_T (x) = q_i (x)$ if $i \in T$, and $q_i \circ \zeta^v_T (x) = v_i$ if $i \not \in T$.
Note that $\zeta^v_{T}={\rm Id}$ if and only if $T=S_0$. It is also clear from the definition that, for $T\subset S$ we have $\psi\circ\phi_T=\zeta^{j(y_0)}_{T\cap S_0}\circ j.$

\begin{lemma}\label{lem:alg}
For two closed points $v$ and $v'$ in $X$, and $\alpha\in\Omega_*(X)$, two cobordism cycles $\zeta^v_{T*}(\alpha)$ and $\zeta^{v'}_{T*}(\alpha)$ are algebraically equivalent.
\end{lemma}

\begin{proof}
Let $X_T=\prod_{i\in T}C_i$ and $pr_T:X\to X_T$ be the projection. Let $l=m-|T|$ and $\{1,2,\ldots,m\}\setminus T=\{a_1,\ldots,a_l\}$. Then, $\zeta^v_{T}=\iota_l\circ\cdots\circ\iota_1\circ p$ where $\iota_j:X_{T\cup\{a_1,\ldots,a_{j-1}\}}\to X_{T\cup\{a_1,\ldots,a_{j}\}}$ be the inclusion of $v_{a_j}$. Similarly, $\zeta^{v'}_{T}=\iota'_l\circ\cdots\circ\iota'_1\circ p$ where $\iota'_j$ is the inclusion of $v'_{a_j}$.

By \cite[Proposition~3.1.9]{LM}, for $X'$ in $\Sm_k$ and a smooth projective curve $C$ with a closed point $\{p\}$, we have $[X'\times\{p\}\to X'\times C]=[\sO_{X'\times C}(X'\times\{p\})]$. For another closed point $\{p'\}$ in $C$, the line bundles $\sO_{X'\times C}(X'\times\{p\})$ and $\sO_{X'\times C}(X'\times\{p'\})$ are algebraically equivalent so that $[X'\times\{p\}\to X'\times C]\sim_\alg [X'\times\{p'\}\to X'\times C]$ by the definition of algebraic equivalence as in \cite{KP}. Thus, for any cobordism cycle $\beta_j\in\Omega^*(X_{T\cup\{a_1,\ldots,a_{j-1}\}})$, we have $\iota_{j*}(\beta) \sim_\alg \iota'_{j*}(\beta)$ for any $j$, thereby implying $\zeta^{v}_{T*}(\alpha)\sim_\alg\zeta^{v'}_{T*}(\alpha)$.
\end{proof}

\begin{definition}
Define $\kappa:=\sum_{l=N}^m{m-N\choose l-N} r_l$, which is nonzero as the $r_l$'s satisfy (S3).
\end{definition}

\begin{lemma}\label{lem:smaller}
The cobordism $1$-cycle $\alpha=[j: Y\to X]\in\Omega_1(X)_\Q$ is smash-equivalent to  a cycle coming from a smaller product of curves. 
\end{lemma}

\begin{proof}
Let $v=j(y_0)$. Then,
\[\psi_*\Delta_c=\sum_{l=1}^m r_l \sum_{T\subset S,\#T=l}\psi_*\Delta_T =\sum_{l=1}^m r_l \sum_{T\subset S,\#T=l}\zeta^{v}_{T\cap S_0*}\alpha.\]
Note that if $T\supset S_0$, $\zeta^{v}_{T\cap S_0}={\rm Id}$ and $\psi_*\Delta_T=0$ if $T\cap S_0=\varnothing$. Let $\sS$ be the set of all subsets of $S$ that intersect $S_0$, $\sU$ be the set of all subsets of $S$ that contain $S_0$. Then,
\begin{multline*}
\psi_*\Delta_c=\sum_{l=1}^m r_l \sum_{T\in\sU,\#T=l}\alpha +\sum_{l=1}^m r_l \sum_{T\in\sS\setminus\sU,\#T=l}\zeta^{v}_{T\cap S_0*}\alpha \\
=\alpha \sum_{l=N}^m r_l{m-N\choose l-N} +\sum_{l=1}^m r_l \sum_{T\in\sS\setminus\sU,\#T=l}\zeta^{v}_{T\cap S_0*}\alpha.
\end{multline*}
Thus, by Corollary~\ref{cor:DeltaSm}, $\alpha+ (1 / \kappa) \cdot \sum_{l=1}^m r_l (\sum_{T\in\sS\setminus\sU,\#T=l}\zeta^{v}_{T\cap S_0*}\alpha)$ is smash-nilpotent. Note that $\zeta^{v}_{T\cap S_0}$ is a projection to a smaller product of curves, followed by an inclusion to $X$. This proves the lemma.
\end{proof}

\begin{theorem}\label{thm:Voev}
Let $\alpha$ be a numerically trivial cobordism $1$-cycle on $X$. Then, $\alpha$ is smash-nilpotent.
\end{theorem}
\begin{proof}
By the degree formula \cite[Theorem 4.4.7]{LM}, we have $\alpha=\sum n_i[j_i: \widetilde{Y_i}\to X] +\sum \gamma_j[\{p_j\}\to X]$,
where $n_i\in\Z$, $\gamma_j\in\bL_1$, $p_j$'s are points  in $X$, $\widetilde{Y_i}$'s are smooth projective curves, and $j_i$ is the composition of a birational morphism $\widetilde{Y_i}\to Y_i$ with the inclusion $Y_i\to X$ for a closed irreducible $Y_i\subset X$. Since $\alpha$ is numerically trivial, $\phi(\alpha)=\sum n_i[Y_i]$ is numerically trivial in $\CH_*(X)$. By Theorem~\ref{thm:num}, $\sum n_i[\widetilde{Y'_i}\to X]$ is numerically trivial for some $\widetilde{Y'_i}$s in $\Sm_k$ with  projective birational morphisms $\pi'_i:\widetilde{Y'_i}\to Y_i$. Using \cite[Lemma~4.5.3]{LM}, we have
$
\sum n_i[\widetilde{Y'_i}\to X]=\sum n_i[j_i: \widetilde{Y_i}\to  X] +\sum \beta_l[\{p'_l\}\to X]. 
$
This implies that $\sum \gamma_j[\{p_j\}\to X]-\sum \beta_l[\{p'_l\}\to X]$ is numerically trivial. Note that $[\{q\}\to X]\sim_\alg [\{p\}\to X]$ for any two points $p$ and $q$ in $X$. Thus, going modulo algebraic equivalence, $\omega[\{p\}\to X]$ is numerically trivial for some $\omega\in\bL_1$, implying that $[\{p\}\to X]$ is numerically trivial. But $\sum \gamma_j[\{p_j\}\to X]\sim_\alg\omega'[\{p\}\to X]$ for some $\omega'\in\bL_1$. Thus, $\sum \gamma_j[\{p_j\}\to X]$ is numerically trivial modulo algebraic equivalence. As we have observed in the proof of Lemma~\ref{lem:Gamma}, a numerically trivial cobordism $0$-cycle is smash-nilpotent. Thus, we only need to show that $\sum n_i[j_i: \widetilde{Y_i}\to X]$ is smash-nilpotent. We already know that, modulo algebraic equivalence, $\sum n_i[j_i: \widetilde{Y_i}\to X]$ is numerically trivial.

We proceed by induction on $N$. Suppose any numerically trivial cobordism $1$-cycle on a product of $l$ curves is smash-nilpotent, for $l<N$. By Lemma~\ref{lem:smaller}, $\sum n_i[j_i: \widetilde{Y_i}\to X]$ is smash-equivalent to 
$
 (1/\kappa) \sum n_i\sum_{l=1}^m r_l \sum_{T\in\sS\setminus\sU,\#T=l}\zeta^{v^i}_{T\cap S_0*}[j_i: \widetilde{Y_i}\to  X],
$
where $v^i=j_i(y^i_0)$, $y^i_0$ being a chosen base point of $\widetilde{Y_i}$. However, by Lemma~\ref{lem:alg}, $\zeta^{v^i}_{T\cap S_0*}[j_i: \widetilde{Y_i}\to X]\sim_\alg \zeta^{v^1}_{T\cap S_0*}[j_i: \widetilde{Y_i}\to X]$. Thus, modulo algebraic equivalence,
$
\sum n_i[j_i: \widetilde{Y_i}\to X]\sim_{\mathrm{smash}} (1/\kappa) \sum_{l=1}^m r_l \sum_{T\in\sS\setminus\sU,\#T=l}\zeta^{v^1}_{T\cap S_0*}\sum n_i[j_i: \widetilde{Y_i}\to X]. 
$
Since $\zeta^{v^1}_{T\cap S_0*}(\sum n_i[j_i: \widetilde{Y_i}\to X])$ is numerically trivial, it is smash-nilpotent by the induction hypothesis. Hence, $\sum n_i[j_i: \widetilde{Y_i}\to  X]$ is smash-nilpotent. It remains to check the cases $N=1,2$. The case of a curve is trivial as algebraic and numerical equivalence coincide for cobordism $1$-cycles on curves, by \cite[Theorem~9.6.(1)]{KP}. Now consider the case where $X$ is a surface. Let $\alpha=\sum n_i[j_i: \widetilde{Y_i}\to X] +\sum \gamma_j[\{p_j\}\to X]$ be a numerically trivial cobordism $1$-cycle in $\Omega^\alg_1(X)_\Q$. Since $\CH^\alg_*(X)$ coincides with $\CH^\num_*(X)$, we have $\phi(\alpha)=\sum n_i[Y_i]=0$, whence \cite[Lemma~4.5.3]{LM} implies, $\sum n_i[j_i: \widetilde{Y_i}\to X]=\sum \beta_l[\{p'_l\}\to X]$. Thus, $\alpha=\omega[\{p\}\to X]\in\Omega^\alg_1(X)_\Q$. This shows $[\{p\}\to X]$ is numerically trivial, hence smash-nilpotent by the argument above. Therefore, $\alpha$ is smash-nilpotent.
\end{proof}

\begin{corollary}\label{cor:voefsab}
Let $Y$ be a smooth projective variety and let $h: X=C_1\times C_2\times\cdots\times C_N\to Y$ be a dominant morphism. Then, numerical equivalence and smash equivalence coincide for cobordism $1$-cycles on $Y$.
\end{corollary}
\begin{proof}
Let $\sL$ be a relatively $h$-ample line bundle on $X$. Let $r:=N-\dim(Y)$ be the relative dimension of $h$. Now, consider $h_*\left(c_1(\sL)^r\right)\in \Omega^0(Y)_\Q$. Note that since $\deg\left(h_*(c_1(\sL)^r)\right)\in\Z$, if $\deg\left(h_*(c_1(\sL)^r)\right)=0$, then $\phi\left(h_*(c_1(\sL)^r)\right)=0\in \CH^0(Y)$, which is not the case since $\sL$ is relatively $h$-ample. Denote $d:=\deg\left(h_*(c_1(\sL)^r)\right)$. Thus, by the degree formula \cite[Theorem 4.4.7]{LM},
\[h_*\left(c_1(\sL)^r\right)=d[{\rm Id}_Y]+\sum_{ \codim_Y Z>0
}\omega_Z[\widetilde{Z}\to Y]\text{ with }\widetilde{Z}\text{ smooth, and birational over }Z.\]
Now, let $\alpha\in\Omega_1(Y)_\Q$ be numerically equivalent to 0. Since $X$ and $Y$ are smooth, $h$ is \lci Thus, we may consider the pullback
$h^*\alpha$. Note that by the projection formula,
\[ 
h_*\left(c_1(\sL)^r\cdot h^*\alpha\right)=h_*\left(c_1(\sL)^r\right)\alpha=d\alpha+\sum_{
\codim_YZ=1, \omega_Z\in\bL_1
}\omega_Z[\widetilde{Z}\to Y]\cdot\alpha.
\]
But, $[\widetilde{Z}\to Y]\cdot\alpha\in\Omega_0(Y)_\Q$ and is numerically trivial. We observed in the proof of Lemma~\ref{lem:Gamma} that a numerically trivial cobordism $0$-cycle on a smooth projective variety is smash-nilpotent. Also, $c_1(\sL)^r\cdot h^*\alpha$ being a numerically trivial cobordism $1$-cycle on $X$, is smash-nilpotent by Theorem~\ref{thm:Voev}. Thus, $h_*\left(c_1(\sL)^r\cdot h^*\alpha\right)$ is smash-nilpotent, which implies $d\alpha$, and hence $\alpha$, is smash-nilpotent since $d\neq 0$.
\end{proof}

\begin{ack}\label{ack}JP was supported by National Research Foundation of Korea (NRF) grant funded by the Korean government (MSIP) (No. 2013042157), Korea Institute for Advanced Study (KIAS) grant funded by the Korean government (MSIP), and the TJ Park Junior Faculty Fellowship funded by POSCO TJ Park Foundation. AB was supported by the postdoctoral associateship of KAIST under the mentorship of JP.

\end{ack}


\begin{thebibliography}{99}
%
%
\bibitem{Andre}
{Y. Andr\'e}, Une introduction aux motives (motifs purs, motifs mixtes, p\'eriodes), {\em Panoramas et Synth\`eses }17 (Soc. Math. France, Paris, 2004). xii+261 pp.

\bibitem{BH}
{A. Banerjee \and  T. Hudson}, Fourier-Mukai transformation on algebraic cobordism, preprint available at \url{http://arxiv.org/abs/1311.4039}, 2013.
%
\bibitem{Clozel}
{L. Clozel}, Equivalence num\'erique et \'equivalence cohomologique pour les vari\'et\'es ab\'eliennes sur les corps finis, {\em Ann. Math. } 150 (1999), 151--163.
\bibitem{CF}
{P. E. Conner \and E. E. Floyd}, Differentiable periodic maps, {\em Bull. Amer. Math. Soc. } 68 (1962), no. 2, 76--86.

\bibitem{Fulton}
{W. Fulton}, Intersection theory, Second Edition, {\em Ergebnisse der Math. Grenzgebiete }3 (Springer-Verlag, Berlin, 1998). xiv+470 pp.

\bibitem{KS}
{B. Kahn \and R. Sebastian}, Smash-nilpotent cycles on Abelian 3-folds, {\em Math. Res. Lett. }16 (2009), no. 6, 1007--1010.

\bibitem{Kimura}
{S-I. Kimura}, Chow groups are finite dimensional in some sense, {\em Math. Ann. }331 (2005), 173--201.

\bibitem{Kleiman}
{S. Kleiman}, The standard conjecture, in Motives (Seattle, WA, 1991), Proceedings of Symposia in Pure Mathematics 55 (1994), American Mathematical Society, 3--20.

\bibitem{KP}
{A. Krishna \and J. Park}, Algebraic cobordism theory attached to algebraic equivalence, {\em J. $K$-theory.} 11 (2013), 73--112.
%
\bibitem{Lazard}
{M. Lazard}, Sur les groupes de Lie formels \`a un param\`etre, {\em Bull. Soc. Math. France}, 83 (1955), 251--274.
%
\bibitem{LM}
{M. Levine \and F. Morel}, Algebraic Cobordism, {\em Springer Monographs Math. } (Springer, Berlin, 2007). xii+244 pp.
%
\bibitem{LP}
{M. Levine \and R. Pandharipande}, Algebraic cobordism revisited, {\em Invent. Math. }176 (2009), 63--130.
\bibitem{Lieberman}
{D. Lieberman}, Numerical and homological equivalence of algebraic cycles on Hodge manifolds, {\em Amer. J. Math. } 90 (1968), no. 2, 366--374.

%
\bibitem{Manin}
{Ju. I. Manin}, Correspondences, motifs and monoidal transformations, {\em Math. USSR-Sb. }6 (1968), 439--470.
%
\bibitem{Mumford}
{D. Mumford}, Abelian Varieties, Oxford University Press, 1974.

\bibitem{NZ}
{A. Nenashev \and K. Zainoulline}, Oriented cohomology and motivic decompositions of relative cellular spaces, {\em J. Pure Appl. Algebra }205 (2006), 323--340.
%

\bibitem{Quick}
{G. Quick}, Stable \'etale realization and \'etale cobordism, {\em Adv. Math. } 214 (2007), 730--760.

\bibitem{Quick2}
{G. Quick}, Torsion algebraic cycles and \'etale cobordism, {\em Adv. Math. }227 (2011), 962--985.

 
\bibitem{Quillen} 
{D. Quillen}, Elementary proofs of some results of cobordism theory using Steenrod operations, {\em Adv. Math. }7 (1971), 29--56.

\bibitem{Schwarzenberger}
{R. L. E. Schwarzenberger}, Jacobians and symmetric products, {\em Illinois J. Math. } 7 (1963), 257--268.

\bibitem{Sebastian}
{R. Sebastian},  Smash nilpotent cycles on varieties dominated by products of curves, {\em Compositio Math. } 149 (2013), 1511--1518. doi:10.1112/S0010437X13007197. 
%
\bibitem{Shermenev}
{A. M. Shermenev}, The motif of an abelian variety, {\em Funct. Anal. Appl.} 8 (1974), 47--53.
\bibitem{Totaro}
{B. Totaro}, Torsion algebraic cycles and complex cobordism, {\em J. Amer. Math. Soc. }10 (1997), 467--493.
%
\bibitem{VY}
{A. Vishik \and N. Yagita}, Algebraic cobordisms of a Pfister quadric, {\em J. London Math. Soc. }76 (2007), no. 2, 586--604.
%
\bibitem{Voevodsky} 
{V. Voevodsky}, A nilpotence theorem for cycles algebraically equivalent to zero, {\em Internat. Math. Res. Notices }(1995), 187--198.
%
\bibitem{Voisin}
{C. Voisin}, Remarks on zero-cycles of self-products of varieties, in Maruyama, Masaki (ed.), Moduli of vector bundles, {\em Lect. Notes in Pure Appl. Math. }179 (Marcel Dekker, New York, 1996) 265--285.
\end{thebibliography}
\end{document}